\documentclass[a4paper,11pt]{amsart}

\usepackage[OT1]{fontenc}
\usepackage[margin=1.25in]{geometry}
\usepackage[american]{babel}
\usepackage[utf8x]{inputenc}
\usepackage[sc]{mathpazo}

\usepackage{amsmath,amssymb,amsfonts,mathrsfs}
\usepackage{hyperref}
\usepackage{amsthm}
\usepackage{graphicx}
\usepackage{varioref}
\usepackage{datetime}
\usepackage{cancel}
\usepackage{mathtools}

\usepackage{booktabs}

\usepackage{tikz-cd}
\usepackage{csquotes}

\usepackage{interval}
\intervalconfig{soft open fences}
\usepackage{dsfont}
\usepackage{microtype}

\usepackage{faktor}
\usepackage{xcolor}

\usepackage{pgfplots}
\pgfplotsset{compat=newest}
\usepgfplotslibrary{groupplots}
\usepgfplotslibrary{polar}
\usepgfplotslibrary{smithchart}
\usepgfplotslibrary{statistics}
\usepgfplotslibrary{dateplot}
\usepgfplotslibrary{ternary}
\usepgfplotslibrary{fillbetween}

\DeclareMathAlphabet{\mathdutchcal}{U}{dutchcal}{m}{n}
\SetMathAlphabet{\mathdutchcal}{bold}{U}{dutchcal}{b}{n}
\DeclareMathAlphabet{\mathdutchbcal}{U}{dutchcal}{b}{n}

\newtheorem{theorem}{Theorem}[section]

\newtheorem{lemma}[theorem]{Lemma}
\newtheorem{corollary}[theorem]{Corollary}

\theoremstyle{definition}
\newtheorem{definition}[theorem]{Definition}
\newtheorem{example}[theorem]{Example}

\theoremstyle{remark}
\newtheorem{remark}[theorem]{Remark}

\numberwithin{equation}{section}

\newcommand{\C}{\mathbb{C}}

\newcommand{\N}{\mathbb{N}}

\newcommand{\R}{\mathbb{R}}
\newcommand{\Z}{\mathbb{Z}}

\DeclarePairedDelimiter\abs{\lvert}{\rvert}
\DeclarePairedDelimiter\norm{\lVert}{\rVert}

\renewcommand{\epsilon}{\ensuremath\varepsilon}

\renewcommand{\phi}{\ensuremath{\varphi}}

\newcommand{\eps}{\epsilon}

\newcommand{\anyfield}{\mathbb{F}}
\NewDocumentCommand\field{e{_}}{\IfValueF{#1}{\anyfield}\IfNoValueF{#1}{\Z_{#1}}}

\newcommand{\Reals}{\R}
\newcommand{\Nats}{\N}
\newcommand{\into}{\hookrightarrow}
\newcommand{\trans}{\pitchfork}
\newcommand{\M}{\mathscr{M}}
\newcommand{\isom}{\cong}

\newcommand{\B}{\mathcal{B}}
\newcommand{\A}{\mathcal{A}}
\newcommand{\Chords}{\mathscr{C}}
\newcommand{\Lag}{\mathscr{L}}

\newcommand{\Tvert}{\Lambda}
\newcommand{\LG}{\mathcal{L}}
\newcommand{\longto}{\longrightarrow}

\newcommand{\floor}[1]{\left\lfloor #1 \right\rfloor}

\DeclareMathOperator{\taut}{taut}
\DeclareMathOperator{\im}{im}

\DeclareMathOperator{\supp}{supp}
\DeclareMathOperator{\Spec}{Spec}
\DeclareMathOperator{\inter}{int}
\DeclareMathOperator{\linspan}{span}

\DeclareMathOperator{\Hof}{Hofer}
\DeclareMathOperator{\Ham}{Ham}

\DeclareMathOperator{\bottle}{bot}

\DeclareMathOperator{\sign}{sign}

\DeclareMathOperator{\osc}{osc}

\definecolor{ceruleanblue}{rgb}{0.16, 0.32, 0.75}

\tikzset{    
    mypoint/.style={
        circle,
        draw,
        inner sep=.3mm
        },  
    whitepoint/.style={
        fill=white, 
        mypoint
        },  
    blackpoint/.style={
        fill=black, 
        mypoint
        },  
    textnode/.style={
        text height=2.5ex, 
        text depth=1ex
        },  
    }

\begin{document}


\title{Hofer geometry of $A_3$-configurations}


\date{\today}
\author{Adrian Dawid}
\address{Department of Pure Mathematics and Mathematical Statistics, University of Cambridge}
\email{apd55@cam.ac.uk}

\begin{abstract} 
    Let $L_0,L_1,L_2 \subset M$ be exact Lagrangian spheres in a Liouville domain $M$ with $2c_1(M)=0$.
    If $L_0,L_1,L_2$ form an $A_3$-configuration, we show that $\Lag(L_0)$ and $\Lag(L_2)$ endowed with the Hofer metric contain 
    quasi-isometric embeddings of $(\Reals^\infty, \norm{\cdot}_\infty)$, i.e.\ 
    infinite-dimensional quasi-flats.
    A corollary of the proof presented here establishes that $\Ham_c(M)$ itself contains an infinite-dimensional quasi-flat.
    We also show that for a Dehn twist $\tau: M \to M$ along $L_1$ the boundary depth of $CF(\tau^{2\ell}(L_0), L')$ is unbounded in $L' \in \Lag(L_2)$ for any $\ell \in \Nats_0$.
\end{abstract}
\maketitle
\section{Introduction and main results}
The large-scale properties of the Hofer metric have been the object of intensive study.
{
The Hofer diameter conjecture asserts that the Hofer diameter of the
Hamiltonian diffeomorphism group is infinite for all symplectic manifolds.
While this question is still open in general, much progress has been made.
A natural refinement of this question is 
the existence of quasi-flats for the Hofer metric.}
A (quasi)-flat is a (quasi)-isometric embedding of a normed vector space into any other metric space.
For the Hofer metric, this question has been extensively studied in the Hamiltonian case.
Numerous examples of finite and infinite-dimensional flats and quasi-flats have been constructed.
There is a natural relative or \textit{Lagrangian} version of the aforementioned classical question.
Namely, one is tempted to ask for which symplectic manifolds $M$ there exists a compact Lagrangian $L \subset M$
for which $\Lag(L) \coloneqq \{\phi(L) \vert \phi \in \Ham_c(M)\}$ contains a quasi-flat when endowed with the Lagrangian Hofer metric.
{
A quite general picture is known for the question of finite-dimensional quasi-flats when $M$ is a tame symplectically aspherical manifold.
This is due to \mbox{Zapolsky}, who showed in~\cite{zapolsky-2013} that quasi-flats\footnote{Zapolsky actually constructs proper \textit{flats} but with respect to an unconventional metric on $\Reals^d$. With respect to the Euclidean metric or $\norm{\cdot}_\infty$ they are quasi-flats.} of finite dimension exist for the Lagrangian Hofer metric
for specific configurations of weakly exact Lagrangians.
Namely, when one has $L,L_1,\dots,L_k \subset M$ with $L_1,\dots,L_k$ pairwise disjoint
such that $L$ intersects every $L_1,\dots,L_k$ in a single transverse point, there exists a $k$-dimensional quasi-flat in $\Lag(L)$.}

In this paper we will instead look at a different setting, namely that of a Liouville domain with an $A_3$-configuration of exact Lagrangian spheres.
This is a very large class of examples. Recall that a generic degree $d$ hypersurface $X \subset \C^n$ with $n \geq 2$ and $d \geq 3$
contains an $A_3$-configuration of Lagrangian spheres (see e.g.~\cite[Sec. III.20]{seidel-2008}).
In this setting we prove the following result:
\begin{theorem}\label{thm:main_1} 
    Let $(M, \omega = d\lambda)$ be a Liouville domain with $2c_1(M) = 0$ 
    and $L_0, L_1, L_2 \subset M$ exact Lagrangian spheres in an $A_3$-configuration. 
    Then there is a map 
    \begin{align*}
        \Phi: (\Reals^\infty, d_\infty) &\into (\Lag(L_2), d_{\Hof})
    \end{align*}
    which is a quasi-isometric embedding with quasi-isometry constant $2$, i.e.\ for any $v,w \in \Reals^{\infty}$
    \[\frac{1}{2}\cdot d_\infty(v,w) \leq d_{\Hof}(\Phi(v),\Phi(w)) \leq 2 \cdot d_{\infty}(v,w).\]
\end{theorem}
Due to the way this theorem is proven, we also obtain the existence of a quasi-flat in the group $\Ham_c(M)$.
As mentioned before, the class of symplectic manifolds which contain an $A_3$-configuration is quite large and contains many examples 
for which the metric properties of $(\Ham_c(M), d_{\Hof})$ have not yet been closely studied.
This includes the aforementioned generic degree $d$ affine hypersurfaces in $\C^n$ with $n \geq 2$ and $d \geq 3$.
Since it is of independent interest, we record this fact in the following corollary:
\begin{corollary}\label{cor:absolute_version}
    Let $(M, \omega = d\lambda)$ be a Liouville domain with $2c_1(M) = 0$ 
    that contains an $A_3$-configuration of exact Lagrangian spheres.
    Then $(\Ham_c(M), d_{\Hof})$ contains a quasi-isometric embedding of $(\Reals^\infty, d_\infty)$.
\end{corollary}

There is a rich body of prior results about these questions in different settings.
On the absolute side, Py showed in~\cite{py-2008} that the group of Hamiltonian diffeomorphisms of certain symplectic manifolds
contains Hofer flats of arbitrary finite dimension.
Later Usher showed that 
the Hamiltonian diffeomorphism group of a closed symplectic manifold contains an infinite-dimensional quasi-flat if 
the manifold admits a nontrivial Hamiltonian vector field all of whose contractible closed orbits are constant~\cite{usher-2013}.
Further, Usher showed in~\cite{usher-2014} that the compactly supported Hamiltonian diffeomorphism group of a cotangent bundle $T^*N$
admits an infinite-dimensional quasi-flat in some cases. In particular, $N$ has to be a closed Riemannian manifold which fulfills certain metric conditions.
For example any $S^n$ with $n \geq 3$ is allowed.
Recently, Polterovich and Shelukhin have constructed infinite-dimensional quasi-flats in certain low-dimensional settings~\cite{polterovich-shelukhin-2023}.

{
On the relative side, there is the aforementioned result of \mbox{Zapolsky} for finite-dimensional quasi-flats. 
For the case of infinite-dimensional flats,~\cite[Theorem 1.5]{zapolsky-2013} asserts the existence of an infinite-dimensional quasi-flat for weakly exact Lagrangians that fiber over $S^1$ in a tame symplectic manifold.
Another class of infinite-dimensional quasi-flats are constructed in~\cite{usher-2013}. 
These come from Lagrangians which are isotopic to the diagonal in the product of certain symplectic manifolds.
Here some dynamic properties have to be fulfilled by the manifold.
}

For the special case of cotangent bundles much more is known.
The related question of the finiteness of the Hofer diameter is known in full generality.
This is due to Gong, who recently showed that the Hofer diameter of the space of Lagrangians which are Hamiltonian isotopic to a cotangent fiber in the disk cotangent bundle $D^*N$ 
is infinite for any closed Riemannian manifold $N$~\cite[Corollary 2]{gong-2023}.
This result is obtained as a corollary of the respective statement for the spectral {metric}, which is covered by the main theorem in~\cite{gong-2023}.
Recently, an infinite-dimensional quasi-flat for the spectral metric has been constructed by Feng and Zhang under certain dynamical conditions~\cite{feng-zhang-2024}.
For the question of Hofer quasi-flats, 
{
Usher showed in~\cite{usher-2014} that there is an infinite-dimensional quasi-flat in the space of Lagrangians which are Hamiltonian isotopic to a cotangent fiber in $T^*N$ through isotopies with compact support arbitrarily close to the zero section. 
In particular, any $S^n$ with $n \geq 3$ is included in the result.}

In this paper, we follow an approach specifically tailored towards spheres $L_0,L_1,L_2 \subset M$ in an $A_3$-configuration. 
Locally, a neighborhood of an $A_3$-configuration looks like a sphere cotangent bundle with two preferred fibers to which handles are attached.
Recall, that three Lagrangian spheres $L_0,L_1,L_2 \subset M$ in a Liouville domain $M$ are in an $A_3$-configuration if $L_0 \cap L_1$ and $L_1 \cap L_2$ contain each a single transverse point and $L_0 \cap L_2 =\emptyset$.
We will use carefully constructed models of a Dehn twist along $L_1$ to mimic the behavior of wrapping in the cotangent bundle case.
For any $d \in \Nats$, an $\Reals^d$-parametrized family of Hamiltonian diffeomorphisms $\phi_v$ is then constructed.
Using spectral invariants associated to classes in $HF(\tau^{2k}(L_0), \phi_v(L_2))$ for large $k \in \Nats$, we show that 
$\Lag(L_2)$ contains Hofer quasi-flats of dimension $d$.
This first step of the proof is easier to handle geometrically. 
Then, by a slight variation of the argument we construct infinite-dimensional quasi-flats.

The method of our proof also applies in the case of $A_2$-plumbings, i.e.\ the plumbing of two copies of $T^*S^n$. We denote this by $A^n_2$ as it is the $A_2$-Milnor fiber which can also be described as the affine variety
\[A_2^n = \{z_0^2 + \cdots + z_{n-1}^2 + z_n^3 = 1\} \subset \C^{n+1}\]
with the restriction of the standard symplectic form $\omega_{\C^{n+1}} = \frac{i}{2\pi} (dz_0 \wedge d\bar{z}_0 + \cdots + dz_n \wedge d\bar{z}_n)$.
{
The zero section of either cotangent bundle defines a Lagrangian sphere $S^n \into A^n_2$, which is identified with a cotangent fiber of the other cotangent bundle near the gluing point.
We denote the resulting Lagrangian spheres by $L^{(1)}_{S^n}$ and $L^{(2)}_{S^n}$ respectively. The result from above can now be adapted to this setting as follows:
\begin{corollary}\label{cor:main_1}
    Let $n \in \Nats$ be arbitrary. Then with the notation from above there is a map 
    \begin{align*}
        \Phi_i: (\Reals^\infty, d_\infty) &\into (\Lag(L_{S^n}^{(i)
        } ), d_{\Hof}) 
    \end{align*}
    for $i \in \{1,2\}$
    which is a quasi-isometric embedding with quasi-isometry constant $2$, i.e.\ for any $v,w \in \Reals^\infty$
    \[\frac{1}{2} \cdot d_\infty(v,w) \leq d_{\Hof}(\Phi_i(v),\Phi_i(w)) \leq 2 \cdot d_{\infty}(v,w).\]
\end{corollary}
For the sake of completeness we note that our method also covers the case of {sphere} cotangent bundles.
Here we recover part of Usher's result in~\cite{usher-2014}, namely the case for sphere cotangent bundles when $n \geq 3$.
\begin{corollary}[Usher 2014]\label{cor:main_2}
    Let $n \geq 2$ be arbitrary. Let $F \subset T^*S^n$ be any fiber.
    Then there exists a map
    \begin{align*}
        \Phi: (\Reals^\infty, d_\infty) &\into (\Lag_c(F), d_{\Hof})
    \end{align*}
    which is a quasi-isometric embedding with quasi-isometry constant $2$, i.e.\ for any $v,w \in \Reals^\infty$
    \[\frac{1}{2} \cdot d_\infty(v,w) \leq d_{\Hof}(\Phi(v),\Phi(w)) \leq 2 \cdot d_{\infty}(v,w).\]
    Here $\Lag_c(F) = \{\phi(F) \mid \phi \in \Ham_c(T^*S^n)\}$.
\end{corollary}}

Lastly, we can obtain a result on the boundary depth with the same methods.
The boundary depth $\beta$ of a Floer complex was introduced by Usher in~\cite{usher-2013}.
From a persistence homology perspective, 
it measures the length of the longest finite bar in a barcode. For Floer homology, this can be thought of as the highest energy of a $J$-holomorpic strip that contributes non-trivially to the Floer differential.
It is closely related to the Hofer norm and the spectral norm but can also be studied independently.
The unboundedness of the boundary depth for two fibers in some cotangent bundles is for example shown in~\cite{usher-2014}.
We can show the following analogous statement using the results which lead of up Theorem~\ref{thm:main_1}.
\begin{theorem}\label{thm:main_3} 
    Let $(M, \omega = d\lambda)$ be a Liouville domain with $2c_1(M) = 0$ 
    and $L_0, L_1, L_2 \subset M$ exact Lagrangian spheres in an $A_3$-configuration. 
    The boundary depth $\beta(L_0,L)$ is unbounded in $L \in \Lag(L_2)$.
    Let $\tau: M \to M$ be a Dehn twist along $L_1$. Then for any $\ell \in \Nats_0$, 
    the boundary depth $\beta(\tau^{2\ell}(L_0), L)$ is unbounded in $L \in \Lag(L_2)$.
\end{theorem}
\begin{figure}[ht] 
    \centering 
    \includegraphics[width=0.95\textwidth]{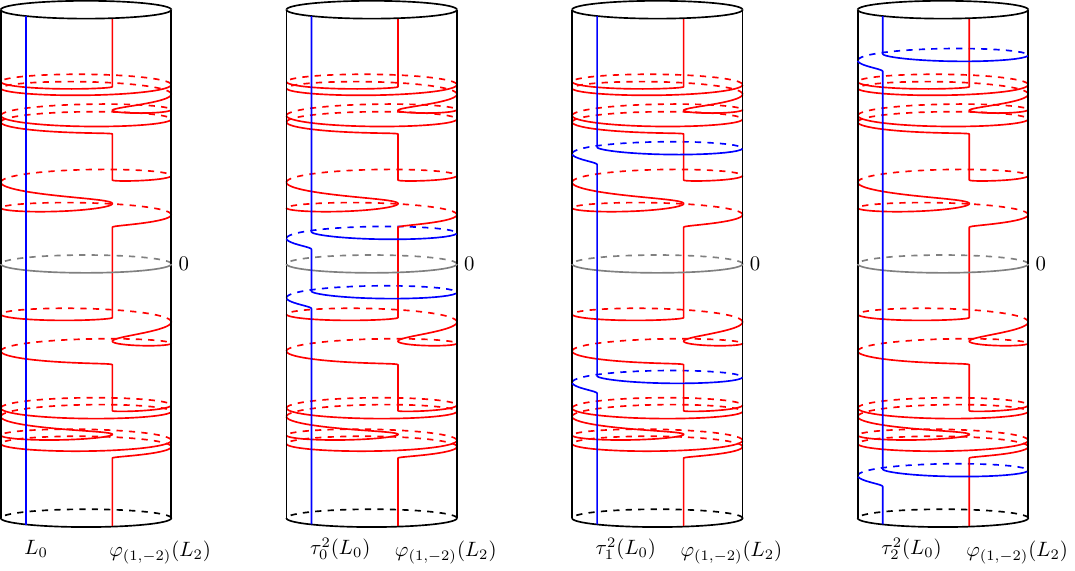} 
    \caption{Illustration of the Dehn twists used in the proof of Theorem~\ref{thm:main_1}.}\label{fig:spectral}
\end{figure}
\subsection{Strategy of the proof}\label{sec:intro_proof}
{
We will now give a very brief outline of the proof strategy for the aforementioned results.
Since it is more natural, we only sketch the construction of finite-dimensional quasi-flats.
The construction of infinite-dimensional quasi-flats in Section~\ref{sec:infty_flats} is only a slight technical variation of this idea.
The proof boils down to the usage of persistence data coming from a Floer complex.

We work in a Weinstein neighborhood $U$ of the middle sphere $L_1 \subset M$.
Henceforth, we further fix an (auxiliary) framing of $L_1$, i.e.\ a smooth identification with the round standard sphere $S^n$.
We regard $L_1$ as endowed with the round metric induced by this framing.
Let $d$ be the dimension of the quasi-flat we wish to construct.
We then split $U$ into $2d+1$ radial shells.
Here \textit{radial} refers to the norm of the tangent vector in $TS^n$ that we can associate 
to any $c \in U$ by using the Weinstein symplectomorphism between $U$ and a neighborhood of the zero-section in $T^*L_1 \cong T^*S^n$.
The musical isomorphism $\sharp: T^*S^n \to TS^n$ associated to the round metric then gives us a tangent vector.
A \textit{radial shell} means all such $c \in U$ whose associated tangent vector in $TS^n$ has a norm in a specific closed interval.

We equip these shells with preferred models of a Dehn twist around $L_1$ such that the square of the Dehn twist 
is supported in the respective shell. These are constructed in Section~\ref{sec:setup}.
We denote these by $\tau_0,\dots,\tau_d$, where $\tau_i^2$ is supported in the $(2i+1)$-th radial shell.

To construct the flat, we associate to an element $v \in \Reals^d$ a Hamiltonian diffeomorphism $\phi_v$ whose support lies 
outside that of any $\tau_i^2$.
The value $v_i$ controls the behavior of $\phi_v$ in the $2i$-th radial shell.
See Section~\ref{sec:setup} for the details of the construction which is based 
on a reparametrization of the cogeodesic flow on $T^*S^n \cong T^*L_1$.
For a generic $v \in \Reals^d$ we have $\phi_v(L_2) \trans \tau^{2k}_i(L_0)$ for $i \in \{0,\dots,d-1\}$ and $k \in \Nats$.
Thus, we can define the Floer complex $CF(\tau^{2k}_i(L_0), \phi_v(L_2))$ without a Hamiltonian perturbation.
By applying an iterate of $\tau_i^2$ to $L_0$ we generate many intersection points in $(2i+1)$-th radial shell. 
Some of these intersection points will define homology classes which do not contain any of the intersection points generated by $\phi_v$, i.e. which also lie in $L_0 \cap \phi_v(L_2)$.
We then look at the spectral invariants of these classes. 
Since they are supported at a single generator, their spectral invariants can be computed relatively easily.
By comparing the spectral invariants obtained this way from $\tau_i$ and $\tau_{i+1}$, we can isolate what happens in the $2(i+1)$-th radial shell.
See Figure~\ref{fig:spectral} for a visualization of the geometric idea.
We denote these differences by $a_i(v)$, see Section~\ref{sec:spectral} for the actual construction.

An index argument given through Lemma~\ref{lem:degrees_phi} and Lemma~\ref{lem:degrees_tau}
establishes that there are certain degrees which can only contain such classes.
We then utilize that
$HF(\tau^{2k}_i(L_0), \phi_v(L_2))$ has at most one class of a given degree for $n \geq 2$ (see~\cite{frauenfelder-schlenk-2005}).
The map induced on Floer homology by a Hamiltonian diffeomorphisms has to preserve these classes.
By utilizing this fact and the Hofer-Lipschitz property of spectral invariants (cf. Corollary~\ref{cor:lipschitz}), we can show that $\abs{a_i(v) - a_i(w)} \leq 2d_{\Hof}(\phi_v(L_2), \phi_w(L_2))$.
Due to the construction of $\phi_v$ and $\phi_w$ we can explicitly compute that $\abs{a_i(v) - a_i(w)} = \abs{v_{i+1} - w_{i+1}}$.
This can be repeated for any $i \in \{0,\dots,d-1\}$, and we obtain Lemma~\ref{lem:lower_bound}, which asserts that
\[\frac{1}{2}\norm{v-w}_{\infty} \leq d_{\Hof}(\phi_v(L_2), \phi_w(L_2)).\]
The upper inequality can be obtained by an easy and direct computation.
However, the resulting inequality is phrased in terms of the $\norm{\cdot}_1$-norm (cf. Lemma~\ref{lem:upper_bound}).
A slightly more involved construction given in Section~\ref{sec:infty_flats} allows us 
to obtain an upper bound with respect to the $\norm{\cdot}_\infty$-norm.}
This complete the proof of Theorem~\ref{thm:main_1}.
\subsection{Structure of the paper}
The paper is organized as follows: The actual construction of the finite-dimensional quasi-flats is contained in Section~\ref{sec:setup}.
The lower bound for the quasi-isometry is proved in Section~\ref{sec:spectral}.
Then, the upper bound for a finite-dimensional version of Theorem~\ref{thm:main_1} is obtained in Section~\ref{sec:flats}.
In Section~\ref{sec:infty_flats}, the infinite-dimensional case is treated separately and the proof of Theorem~\ref{thm:main_1} is given.
The proof of Theorem~\ref{thm:main_3} is given in Section~\ref{sec:boundary_depth}.
The action and index computations contained in Sections~\ref{sec:action_comp} and~\ref{sec:index_comp} are used throughout these proofs.
The purpose of Section~\ref{sec:prelims} is to fix notation and conventions as the actual proofs are computational and rely on explicit choices of gradings and primitives. 
All the material contained in this section is standard. However, for the convenience of the reader we include this detailed review.
An expert reader might wish to start with Section~\ref{sec:setup} instead and refer to Section~\ref{sec:prelims} only as needed.
\subsection{Acknowledgments}
I would like to thank Paul Biran for supervising the Master's thesis out of which the idea for this paper originally grew and many illuminating conversations about the topics discussed therein.
I would also like to thank Ivan Smith and Jack Smith for many invaluable conversations
and Wenmin Gong for discussing his work on the unboundedness of the Hofer norm on cotangent fibers with me.
Finally, I thank the anonymous referee for helpful comments and
corrections. 
During the writing of this paper, I was supported by EPSRC grant EP/X030660/1.

\section{Preliminaries}\label{sec:prelims} 
Here we want to briefly recall the relevant concepts from persistent homology and Floer theory for this paper.
We refer the reader to~\cite{seidel-2007, seidel-2008, fukaya-lagrangian-2010} for a detailed treatment of the Floer homology side and to~\cite{polterovich-et-al-2020} and~\cite{usher-2008} for a comprehensive overview of persistent homology theory.
All results in this section are well-known in the field and are just restated to either fix conventions or for the convenience of the reader.
\subsection{Persistent Homology}\label{sec:persistence}
In the following section we will give a brief overview of persistence modules and persistent homology. The exposition follows~\cite{polterovich-et-al-2020}.
For the following we fix any field $\field$,
even though we will later always work with $\field = \field_2$.

We will now quickly introduce the definition of a persistence module, as well as a metric on the space of (isomorphism classes of) persistence modules.
\begin{definition}\label{def:persistence_module}
    A persistence module is a pair $(V, \pi)$, where $V$ is a collection ${\{V_t\}}_{t\in\Reals}$ of finite-dimensional $\field$-vector spaces and 
    ${\{\pi_{s,t}: V_s \to V_t \}}_{s\leq t}$ is a 
    collection of linear maps  which fulfill these conditions:
    \begin{enumerate}
        \item For any $s \leq t \leq r$ the following diagram commutes 
                \begin{center}
                    \begin{tikzcd}
                    V_s \arrow[rrr, "\pi_{s,t}"] \arrow[rrrrrr, "{\pi_{s,r}}", bend left] &  &  & V_t \arrow[rrr, "\pi_{t,r}"] &  &  & V_r.
                    \end{tikzcd}
                \end{center}
            \item For any $t \in \Reals$ there exists an $\eps > 0$ such that for any $t - \eps \leq s \leq t$ the map $\pi_{s,t}$ is an isomorphism of 
                $\field$-vector spaces.
            \item The set $\Spec V \subset \Reals$ defined by 
                \begin{align*}
                    \Spec V \coloneqq \{&t \in \Reals \mid \\
                        &\forall \eps > 0 \exists s,r \in \interval{t-\eps}{t+\eps}: \pi_{s,r} \text{ not an isomorphism}\}
                \end{align*}
                is closed, discrete and bounded from below.
    \end{enumerate}
\end{definition}
\begin{remark}\label{rmk:filtered_cc}
    The example which is most relevant for us is the following: Let $(C^\bullet,\partial_C^\bullet)$ be an $\Reals$-filtered finitely generated chain complex.
    For any $\lambda \leq \mu$, denote the natural inclusion of $C^\lambda$ into $C^\mu$ by $i_{\lambda,\mu}: C^\lambda \to C^\mu$.
    Then we can set $V_s \coloneqq H_*(C^s,\partial_C^s)$ for any $s \in \Reals$.
    Let $\pi_{s,t} \coloneqq (i_{s,t})_*: H_*(C^s,\partial_C^s) \to H_*(C^t,\partial_C^t)$.
    Then $(V,\pi)$ has the structure of a persistence module.
    We denote $(V,\pi)$ by $H_*{(C,\partial_C)}^\bullet$.
\end{remark}
\begin{definition}
    Let $(V,\pi)$ and $(V',\pi')$ be persistence modules. Then a morphism of persistence modules $A: (V,\pi) \to (V',\pi')$ is a family $A_t: V_t \to V'_t$
    of linear maps such that for any $s \leq t$ this diagram commutes:
    \begin{center}
        \begin{tikzcd}
            V_s \arrow[rr, "{\pi_{s,t}}"] \arrow[dd, "A_s"'] &  & V_t \arrow[dd, "A_t"] \\
                                                            &  &                       \\
            V'_s \arrow[rr, "{\pi'_{s,t}}"]                  &  & V'_t                 
        \end{tikzcd}
    \end{center}
\end{definition}
\begin{definition}
    Let $(V,\pi)$ be a persistence module and let $\delta \in \Reals$. Then we define the $\delta$-shift of $V$ denoted by $(V[\delta],\pi[\delta])$ to be $V{[\delta]}_t = V_{t+\delta}$
    and ${\pi[\delta]}_{s,t} = \pi_{s+\delta,t+\delta}$. Note that if $\delta \geq 0$ there is a canonical morphism $V \to V[\delta]$ given by $\pi_{t,t+\delta}$.
    Given a morphism $F: V \to W$ we denote by $F[\delta]: V[\delta] \to W[\delta]$ the corresponding morphism on the shifted persistence modules.
\end{definition}
\begin{definition}\label{def:interleaving}
    Let $V,W$ be persistence modules.
    Given $\delta > 0$ a pair of morphism $F: V \to W[\delta], G: W \to V[\delta]$ is called $\delta$-interleaving if the following diagrams commute:
    \begin{center}
    \begin{tikzcd}
        V \arrow[r] \arrow[r, "F"'] \arrow[rr, bend left=49] & {W[\delta]} \arrow[r, "{G[\delta]}"'] & {V[2\delta]} &  & W \arrow[rr, bend left=49] \arrow[r, "G"'] & {V[\delta]} \arrow[r, "{F[\delta]}"'] & {W[2\delta]}
    \end{tikzcd}
    \end{center}
    Here the unlabeled arrows are the canonical morphism associated with a non-negative shift of a persistence module. We call $V$ and $W$ $\delta$-interleaved.
\end{definition}
\begin{definition}
    Given two persistence modules $V,W$ we define 
    \[d_{\inter}(V,W) = \inf \{\delta > 0 \mid V,W \text{ are }\delta\text{-interleaved}\}.\]
    We call $d_{\inter}$ the interleaving distance.
\end{definition}
\noindent Another perspective on persistence modules is given by barcodes. Barcodes can be seen as a particularly useful representation of persistence modules.
\begin{definition}\label{def:barcode}
    A barcode $\mathcal{B}$ is a countable collection of intervals $\interval[open left]{a}{b}$ with $-\infty < a < b \leq +\infty$, called bars, with multiplicities such that 
    \begin{enumerate}
        \item for every $c \in \Reals$ there exists a neighborhood of $c$ which intersects only finitely many bars when counted with multiplicities; and
        \item the set of all endpoints of bars is a closed discrete subset of $\Reals$ bounded from below.
    \end{enumerate}
\end{definition} 
\begin{remark}
    It is possible to also consider \textit{graded} barcodes. In that case any bar would come equipped with an integer grading, i.e.\ an integer associated to it.
    These barcodes are usually used when working with graded homology theories. While we will be considering the graded Floer homology later, we do not need 
    this information on the barcodes.
\end{remark}
\begin{example}
    Let $-\infty < a < b \leq +\infty$ be a pair of real numbers. Then for any $m \in \Nats_{\geq 1}$ we can define the persistence module $\field^m(a,b]$ by
    \[\field^m(a,b] = \begin{cases} \field^m & \text{if } t\in (a,b] \\ 0 & \text{otherwise} \end{cases}\]
        and the maps $\pi_{s,t}$ are given by the identity for $s,t \in (a,b]$ and the zero map otherwise.
\end{example}
We can clearly see that we can define a persistence module based on any barcode. We can just take a direct sum of $\field^m(a,b]$ for all bars $(a,b]$ 
with $m$ being the respective multiplicity. The conditions in Definition~\ref{def:barcode} will then exactly imply the conditions in Definition\ \ref{def:persistence_module} 
that are necessary.
The normal form theorem
(cf.~\cite[Theorem 2.1.2]{polterovich-et-al-2020}) tells us that up to isomorphism any persistence module is represented by a barcode in this way.
Henceforth, we will denote the barcode associated to a persistence module $(V, \pi)$ by $\mathcal{B}(V)$.
We now introduce a metric on barcodes, which is dual to the interleaving distance.
\begin{definition}
    Let $\mathcal{B},\mathcal{C}$ be barcodes. Then a matching $\mu$ of $\mathcal{B}$ and $\mathcal{C}$ is a collection 
    of pairs ${\{(A,B) \in \mathcal{B} \times \mathcal{C}\}}_{i \in I}$ where the number of occurrences of any interval is at 
    most its multiplicity. Given intervals $A \in \mathcal{B},B \in \mathcal{C}$ 
    we call them matched if $(A,B) \in \mu$ and unmatched else. It is important that we count with multiplicities here, 
    i.e.\ if $A \in \mathcal{B}$ has multiplicity $n$ we treat it as $n$ distinct copies for the purpose of it being matched.
\end{definition}
To make the notion slightly more concise we say that a bar is matched in general if it is matched to any bar of the other barcode. 
Of course, we again have to treat a bar with multiplicity $n$ as $n$ distinct copies of itself for the purpose of it being matched.
\begin{definition}\label{def:delta_matching}
    Let $\mathcal{B},\mathcal{C}$ be barcodes and $\mu$ a matching of $\mathcal{B}$ and $\mathcal{C}$.
    Then $\mu$ is called a $\delta$-matching for $\delta > 0$ if the following conditions are met:
    \begin{enumerate}
        \item For any intervals $\interval[open left]{a}{b} = I \in \mathcal{B},\interval[open left]{c}{d}= J \in \mathcal{C}$ that are matched in $\mu$ we have 
        \[ \abs{a-c} \leq \delta  \text{ and }\abs{b-d} \leq \delta.  \]
        \item For any interval $\interval[open left]{a}{b} = I \in \mathcal{B}$
        that is unmatched in $\mu$ we have $b-a \leq 2\delta$.
        \item For any interval $\interval[open left]{a}{b} = I \in \mathcal{C}$
        that is unmatched in $\mu$ we have $b-a \leq 2\delta$.
    \end{enumerate}
\end{definition}
\begin{definition}
    Let $\mathcal{B},\mathcal{C}$ be barcodes. We define 
    \[d_{\bottle}(\mathcal{B},\mathcal{C}) \coloneqq \inf \{\delta > 0 \mid \exists \delta\text{-matching of }\mathcal{B}\text{ and }\mathcal{C}\},\]
    which we call bottleneck distance.
\end{definition}
As mentioned before, the bottleneck distance can be seen as a version of the interleaving distance.
This fact can be expressed as the following isometry theorem, a proof of which can be found in~\cite[Ch. 3]{polterovich-et-al-2020}.
\begin{theorem}\label{thm:persistence_isometry}
    Let $V,W$ be two persistence modules, then 
    \[d_{\inter}(V,W) = d_{\bottle}(\B(V),\B(W)).\]
\end{theorem}
An intrinsic measurement for a barcode is to consider the length of its longest finite bar.
This is precisely half the bottleneck distance between the barcode and its infinite bars.
This notion was introduced (in a symplectic context) in~\cite{usher-2013} but can be naturally phrased in the language of barcodes.
\begin{definition}\label{def:boundary_depth}
    Let $\B$ be a barcode. Then the \textit{boundary depth} of $\B$, denoted $\beta(\B)$, is given by the longest length of a finite bar in $\B$.
    If $\B$ has no finite bars, we set $\beta(\B) = 0$.
    Let $(V,\pi)$ be a persistence module. Then we use the shorthand $\beta(V) \coloneqq \beta(\B(V))$ for the boundary depth of the associated barcode.
\end{definition}
The concept of spectral invariants in Floer theory is much older and has been studied in various settings.
See e.g.~\cite{schwarz-2000,oh-2005, leclercq-2008, usher-2008}.
Our definition below is a pure persistence reformulation.
We refer the reader to~\cite{polterovich-et-al-2020} for more information on the abstract concept and to~\cite{usher-2008} for more information 
on spectral invariants in Floer theory.
To define spectral invariants, we first have to introduce the concept of the terminal vector space of a persistence module.
Let $(V,\pi)$ be a persistence module. Note that ${\{V_t\}}_{t \in \Reals}$ together with ${\{\pi_{s,t}\}}_{s \leq t}$ form a direct system.
We denote its limit by $V_\infty \coloneqq \varinjlim V_t$, which is called the terminal vector space of $(V,\pi)$.
\begin{definition}\label{def:spectral_invariant}
    Let $(V,\pi)$ be a persistence module and $\alpha \in V_\infty$.
    Then 
    \[c(\alpha) \coloneqq \inf \{t \in \Reals \mid \alpha \in \im V_t \to V_\infty\}\]
    is called the spectral invariant of $\alpha$.
\end{definition}
\begin{remark}
    The spectral invariant of $\alpha \in V_\infty$ can also be though of naturally as the left endpoint 
    of the infinite-length bar in $\B(V)$ that represents $\alpha$.
\end{remark}
\subsection{Lagrangian Floer theory} \label{sec:floer_theory_setup}
Henceforth, all symplectic manifolds and their Lagrangian submanifolds will be implicitly assumed to be connected. We assume that Lagrangians 
are always properly embedded. All Hamiltonian functions and diffeomorphisms will be implicitly assumed to be compactly supported away from the boundary unless explicitly stated otherwise.
Given a Hamiltonian diffeomorphism $\phi$ on a symplectic manifold $(M,\omega)$ we will denote by $\mathcal{H}(\phi)$ the set of compactly supported Hamiltonian functions $H: [0,1] \times M \to \Reals$ such that the time-1 
flow of the Hamiltonian vector field $X_H$, given by $\omega(X_H,\cdot) = -dH$, is $\phi$. The group of all compactly supported Hamiltonian diffeomorphisms on $M$ is denoted by $\Ham_c(M)$.
\begin{definition}\label{def:osc}
    Let $H \in C^{\infty}([0,1] \times M,\Reals)$ be a Hamiltonian. Then the oscillation of $H$ is given by 
    \[ \osc (H) \coloneqq \int_0^1 \left[ \max_{x \in M} H(t,x) - \min_{x \in M} H(t,x) \right] dt.\]
\end{definition}
\begin{definition}\label{def:hofer_norm}
    Let $\phi: M \to M$ be a Hamiltonian diffeomorphism compactly supported away from the boundary. Then the Hofer norm of $\phi$ is given by 
    \[\norm{\phi}_{\Hof} \coloneqq \inf_{H \in \mathcal{H}(\phi)} \osc (H).\]
\end{definition}
\begin{definition}\label{def:Lag_hofer_norm}
    Let $L_0,L_1 \subset M$ be Lagrangian submanifolds such that there exists a $\phi \in \Ham_c(M)$ such that $\phi(L_0) = L_1$.
    In this case we call 
    \[d_{\Hof}(L_0,L_1) \coloneqq \inf \{\norm{\phi}_{\Hof} \mid \phi \in \Ham_c(M), \phi(L_0) = L_1\}\]
    the Hofer distance between the Lagrangians $L_0$ and $L_1$.
    If such a $\phi$ does not exist, we set $d_{\Hof}(L_0,L_1) = \infty$.
\end{definition}
We now fix some sign conventions. This is necessary since multiple conventions can be found, under the same names, in the literature.
\begin{definition}\label{def:liouville-vf}
    Let $(M, \omega=d\lambda)$ be an exact symplectic compact manifold, then the vector field $X_\lambda \in \Gamma(TM)$ defined by
    $\omega(\cdot, X_\lambda) = \lambda$ is called the Liouville vector field on $M$.
\end{definition}
\begin{definition}\label{def:liouville-dom}
    An exact compact symplectic manifold $(M, \omega=d\lambda)$ with a compatible almost-complex structure $J:TM \to TM$ is called Liouville domain if the 
    Liouville vector 
    field $X_\lambda$ points transversally inwards on the boundary $\partial M$.
\end{definition}
After this quick recap of some basic notions, we will now lay out the version of Lagrangian Floer theory we use in this text.
Our setup is almost completely identical to that of~\cite[Sec. III.8]{seidel-2008} and closely follows~\cite[Section 2.2.2]{biran-cornea-2022}.
There are however two major differences:
\begin{enumerate}
    \item We use \textit{homological} rather than cohomological conventions, as are used in~\cite{seidel-2008}.
    \item We will not ignore the grading of the homology groups, as is done in~\cite{biran-cornea-2022}. We use graded Lagrangians for this purpose, as introduced in~\cite{seidel-2000}.
\end{enumerate}

Let $(M, \omega = d\lambda)$ be a Liouville domain with a fixed primitive $\lambda$. Further, assume that $2c_1(M) = 0$. Denote by 
$\mathscr{J}$ the set of all $\omega$-compatible almost-complex structures on $M$. Let $L_0,L_1 \subset M$ be compact exact Lagrangian submanifolds. Exactness is always taken to be with respect to the fixed primitive. If $M$ has boundary, we require that there exists a neighborhood of the boundary that contains no intersection points of $L_0$ and $L_1$. A pair $(H,J)$ with $H \in C^{\infty}([0,1]\times M, \Reals)$ and $J \in C^\infty([0,1], \mathscr{J})$ is called
a \textit{Floer datum} if $\phi(L_0) \trans L_1$, $\phi(L_0) \cap L_1 \cap \partial M = \emptyset$, and $\supp H \Subset \inter(M)$.
We further fix primitives 
\[h_{L_k}: L_k \to \Reals\]
for $\lambda\vert_{L_k}$ for $k = 0,1$.
Then we can define the set 
\[\mathscr{P}(L_0,L_1) = \{y \in C^{\infty}([0,1],M) \mid y(0) \in L_0, y(1) \in L_1\}\]
of paths from $L_0$ to $L_1$ endowed with the $C^\infty$-topology.
We focus on the subset
\[\Chords(L_0,L_1; H) = \{y \in \mathscr{P}(L_0,L_1) \mid \dot y(t) = X_H(t,y)\},\]
whose elements are called \textit{Hamiltonian chords}. We usually denote this by $\Chords(L_0,L_1)$ if there is no confusion about the Hamiltonian $H$.
Recall the \textit{Floer equation} for a map $u \in C^{\infty}(\Reals \times [0,1], M)$:
\begin{align} \label{eq:Floer_PDE}
    \begin{cases} \partial_s u + J(t,u)\partial_t u = J(t,u) X_H(t,u) & \\ 
        u(s,0) \in L_0, u(s,1) \in L_1 & \forall s \in \Reals.
    \end{cases}
\end{align}
A solution to~\eqref{eq:Floer_PDE} is called a \textit{Floer trajectory}. We define the energy of such a solution $u: \Reals \times [0,1] \to M$
as 
\begin{equation}\label{eq:floer_energy} E(u) = \int_{\Reals \times [0,1]} \norm{\partial_s u}^2 dt ds, \end{equation}
where $\norm{\cdot}$ is the norm induced by $\omega(\cdot, J(t)\cdot)$.
We will only consider solutions with finite energy.
Assume we have a solution $u$ of~\eqref{eq:Floer_PDE} such that $u(s,\cdot)$ converges uniformly as $s \to \pm \infty$.
Given a pair $y_0,y_1 \in \mathcal{C}(L_0,L_1)$ we denote the set of finite energy solutions of~\eqref{eq:Floer_PDE} with 
\begin{align} \label{eq:Floer_boundary_cond}
    \begin{cases}
        \lim_{s\to+\infty} u(s,t) = y_1 & \text{uniformly in }t\\
        \lim_{s\to-\infty} u(s,t) = y_0 & \text{uniformly in }t
    \end{cases}
\end{align}
by $\M(y_0,y_1)$. If $y_0 \not= y_1$
the moduli space $\M(y_0,y_1)$ admits a free $\Reals$-action given by shifting the $s$-coordinate.
We write
\[\M^*(y_0,y_1) \coloneqq \faktor{\M(y_0,y_1)}{\Reals}\]
for the quotient space under the aforementioned action. We also set $\M^*(y_0,y_0) = \emptyset$ as a matter of convention.
It is a standard fact in the field --- see e.g.~\cite[Ch. 2]{fukaya-lagrangian-2010} ---, that there is a dense subset $\mathscr{J}_{reg} \subset C^\infty([0,1],\mathscr{J})$ of $\omega$-compatible almost-complex structures for which $\M^*(y_0,y_1)$
is either the empty set or 
admits a Gromov-Floer compactification that is a smooth manifold (possibly with connected components of different dimensions) for any $y_0,y_1 \in \mathcal{C}(L_0,L_1)$.
An almost-complex structure in this set is called regular.
We extend the notion of regularity to Floer data. We say that a Floer datum $\mathcal{D} = (H,J)$ is \textit{regular} if the almost-complex structure $J$ 
is regular with respect to $H$ in the above sense.
Let us fix a regular Floer datum $(H,J)$ and define
\[CF(L_0,L_1) \coloneqq \bigoplus_{y \in \Chords(L_0,L_1)}\field_2 \cdot y 
= \linspan_{\field_2} \{y \in \Chords(L_0 , L_1) \},\]
which we call the \textit{Floer complex}.

We will work with graded Floer homology.
{For the detailed construction we refer the reader to~\cite[Sec. II.11]{seidel-2008}
and the other aforementioned literature.
We will later use explicit index computations as a vital part of the main proof.
These are carried out in settings where the Lagrangians intersect transversally.
We therefore quickly review the grading for this setting.}
We follow the framework of graded Lagrangians established in~\cite{seidel-2000}.
Let us recall some notation: Denote by $\LG \to M$ the Lagrangian Grassmannian bundle, i.e.\ the bundle whose fiber $\LG_x$ over $x \in M$ is the 
Lagrangian Grassmannian $\LG(T_x M, \omega_x)$. 
An $\infty$-fold Maslov covering is a covering $\LG^\infty \to \LG$ whose restriction over $\LG_x$ is isomorphic to the universal cover 
$\widetilde{\LG_x}$ for all $ x\in M$.
Such a covering always exists if $c_1(M)$ is $2$-torsion and the equivalence classes of such $\infty$-fold Maslov coverings
form an affine space over $H^1(M; \Z)$, see e.g.~\cite[Lemma 2.2]{seidel-2000}.

In the following, fix an $\infty$-fold Maslov covering $\LG^\infty \to \LG$.
An $\LG^\infty$-grading (henceforth simply called grading) of a Lagrangian $L \subset M$ is a lift $\tilde{L}: L \to \LG^\infty$
of the canonical section $s_L: L \to \LG$ given by $x \mapsto T_xL$.
If the Lagrangian intersects the boundary, a lift away from the boundary is sufficient.
Henceforth, we will assume that our Lagrangians $L_0,L_1 \subset M$ come equipped with gradings $\tilde{L}_0,\tilde{L}_1$.
Further, assume that $L_0 \trans L_1$.
We can now use the Maslov index for paths to fix an absolute grading of the Floer complex.
For any $x \in L_0 \cap L_1$, we obtain 
$\tilde{L}_0(x), \tilde{L}_1(x) \in \LG^\infty$.
Now let $\tilde\lambda_0,\tilde\lambda_1: \interval{0}{1} \to \LG^\infty$ be two paths with 
\begin{align*}
\tilde\lambda_0(0) &= \tilde\lambda_1(0) & \tilde{\lambda}_0(1) &=\tilde{L}_0(x) & \tilde{\lambda}_1(1) &=\tilde{L}_1(x).
\end{align*}
Denote the projections of these paths to $\LG$ by $\lambda_0, \lambda_1$.
Then we define the index of $x$ to be 
\begin{equation} \label{eq:abs_grading}
    \mu(x;\tilde L_0,\tilde L_1) \coloneqq\frac{\dim M}{4}-\mu(\lambda_0,\lambda_1) , \end{equation}
where $\mu(\lambda_0,\lambda_1)$ is the Maslov index for paths as defined in~\cite{robbin-salamon-1993}.
From the properties of the Maslov index, it is easy to see that this is independent of all auxiliary choices (i.e.\ the paths and an identification $T_x M \cong \Reals^{2n}$). We point an interested reader to~\cite{seidel-2000} for details.
\begin{theorem}[Floer]
In the above setting we can define the boundary operator $\partial$ by linear extension of
\begin{align}\label{eq:Fl_boundary_op}
    \partial y_0 = \sum_{\substack{y_1 \in \Chords(L_0,L_1) \\ \mu (y_1; \tilde L_0,\tilde L_1) = \mu (y_0;\tilde L_0,\tilde L_1) - 1}} \sharp \M^*(y_0,y_1)y_1 \mod 2
\end{align}
for any $y_0 \in \Chords(L_0,L_1)$.
Then $\partial \circ \partial = 0$ holds.
\end{theorem}
Given this statement, $(CF(L_0,L_1),\partial)$ is a chain complex. We denote its homology by $HF(L_0,L_1) \coloneqq H_*(CF(L_0,L_1),\partial)$.
Since the almost-complex structure will not be important for us, we mostly suppress it from the notation. Therefore, if the Hamiltonian in the Floer datum vanishes
we write $HF(L_0,L_1)$. Otherwise, we use $HF(L_0,L_1; \mathcal{D})$ or $HF(L_0,L_1; H)$.
\subsection{Action filtration}\label{sec:action_filtration}
We keep the setting of the previous section.
We will now briefly review the natural persistence structure of the Floer complex.
It is well-known that the symplectic action functional induces a filtration on $CF(L_0,L_1)$.
Indeed, this filtration makes Floer homology into a persistence module.
\begin{definition}\label{def:action}
    The functional
    \begin{align*}
        \mathcal{A}: \mathscr{P}(L_0,L_1) &\to \Reals \\
                            y &\mapsto  \int_0^1 H(t,y(t)) dt - \int_0^1 y^*\lambda + h_{L_1}(y(1)) - h_{L_0}(y(0))
    \end{align*}
    is called (symplectic) action functional.
    Note that this functional depends on the choice of primitives on $L_0, L_1$.
\end{definition}
Let $u \in \M(y_0,y_1)$ be arbitrary for some $y_0,y_1 \in \Chords(L_0,L_1; H)$.
Recall that there is a bound on the energy~\eqref{eq:floer_energy} by the action:
$E(u) =  \mathcal{A}(y_0) - \mathcal{A}(y_1). $
Since clearly $0 \leq E(u)$ and equality occurs if and only if $u$ is constant in the first variable, we obtain 
\[\M^*(y_0,y_1) \not= \emptyset \implies \mathcal{A}(y_1) < \mathcal{A}(y_0).\]
By~\eqref{eq:Fl_boundary_op}, this shows that the boundary map is \enquote{action decreasing}.
It is convenient to extend $\A$ to the whole chain complex $CF(L_0,L_1)$.
We define 
\begin{align*}
    \A: CF(L_0,L_1) &\to \Reals \\
    \sum_{i=0}^k a_i y_i &\mapsto \max \left\{ \A(y_i) \mid a_i \not= 0\right\},
\end{align*}
where the left side is a sum of Hamiltonian chords with $\field_2$-coefficients.
Clearly, the boundary is still action decreasing with respect to this extension.

This action now gives us a filtration of the chain complex.
Formalizing this idea, for any $\lambda \in \Reals$ we can define 
\[CF(L_0,L_1)^\lambda \coloneqq \linspan_{\field_2} \langle y \in \Chords(L_0, L_1) \mid \mathcal{A}(y) < \lambda \rangle.\]
Our considerations from above imply that $\partial$ is compatible with the filtration of the Floer complex.
Thus, $CF(L_0,L_1)$ is an $\Reals$-filtered complex and by Remark~\ref{rmk:filtered_cc} we obtain the persistence module
$HF{(L_0,L_1)}^\bullet$.
Note that the filtration does not depend on the choice of almost-complex structure.
Thus, we can distill all of our setup into the following definition:
\begin{definition}
    Let $(M, \omega = d\lambda)$ be a Liouville domain and $L_0,L_1 \subset M$ compact exact Lagrangian submanifolds with fixed primitives for $\lambda$ as above. 
    Further, let $\mathcal{D} = (H,J)$ be a regular Floer datum as described above. 
    Then we denote by $HF{(L_0,L_1; \mathcal{D})}^\bullet$ the Floer persistence module associated to this data. We further denote its barcode by $\mathcal{B}(L_0,L_1; \mathcal{D})$.
\end{definition}
\begin{remark}
    {Our Lagrangians $L_0, L_1 \subset M$ are also not necessarily closed. 
    However, we do not work with a wrapped non-compact flavor of Floer theory. 
    Recall that for a Floer datum 
    $\mathcal{D} = (H,J)$ to be regular, $\phi_H(L_0) \cap L_1 \cap \partial M = \emptyset$ must hold.
    Since by assumption $L_0$ and $L_1$ do not intersect in a sufficiently small neighborhood of the boundary, 
    this implies that $\phi_H(L_0) \cap L_1$ is a finite set which is itself contained in a 
    compact part of the interior of $M$.
    Since $H$ is also assumed to be compactly supported (inside the interior of $M$), 
    everything of Floer-theoretic interest takes place in a compact part of $M$ and away from the boundary.}
    {Furthermore, all of this not relevant for the proof of Theorem~\ref{thm:main_1}.
    Here all Lagrangians are closed since they are Lagrangian spheres.
    To show Corollary~\ref{cor:main_1} and Corollary~\ref{cor:main_2},
    we need to allow $M = A_2^n$ or $M = T^*S^n$ as the ambient manifold. 
    These are Weinstein domains which can be obtained from completing a disk cotangent bundle $D^*S^n$ and the plumbing $D^*S^n \# D^*S^n$
    of two disk cotangent bundles respectively.
    We only need the Floer homology $HF(L,L')$, where $L$ and $L'$ are images of fibers under Dehn twists or compactly supported Hamiltonian diffeomorphisms.
    For the cotangent bundle, this is exactly the setting of~\cite{seidel-2004}, and we also refer the reader to this paper and the references therein.
    As is standard in this kind of non-compact setting, an appropriate (generic) choice of almost complex structure is made to ensure 
    that pseudo-holomorpic curves do not escape the compact part of the manifold.
    See e.g.~\cite[Theorem 2.1]{oh-2001}
    for details.
    }
\end{remark}
\subsection{Filtered Continuation and Naturality Maps}\label{sec:naturality}
As noted above for two regular Floer data $\mathcal{D}_0 = (H^0,J^0), \mathcal{D}_1 = (H^1,J^1)$ the Floer homology $HF(L_0,L_1)$ does not depend on the Floer datum used to define it.
This is formalized by the existence of a chain map $\Psi_{\mathcal{D}_0,\mathcal{D}_1}: CF(L_0,L_1; \mathcal{D}_0) \longto CF(L_0,L_1; \mathcal{D}_1)$.
This is a quasi-isomorphism canonical up to chain homotopy and induces a (canonical) isomorphism on homology:
\[
    H(\Psi_{\mathcal{D}_0,\mathcal{D}_1}): HF(L_0,L_1; \mathcal{D}_0) \longto HF(L_0,L_1; \mathcal{D}_1).
\]
The chain map $\Psi_{\mathcal{D}_0,\mathcal{D}_1}: CF(L_0,L_1; \mathcal{D}_0) \longto CF(L_0,L_1; \mathcal{D}_1)$ is called \textit{continuation map}.
This construction is by now standard.
A filtration aware version of the continuation map is given in~\cite[Prop. 6.1]{usher-2013}. 
Detailed computations of the energies involved are given in~\cite{fukaya-oh-ohta-ono-2011}. 
\begin{lemma}\label{lem:continuation_action}
    Let $\mathcal{D}_0 = (G,J)$ and $\mathcal{D}_1 = (H,J')$ be two regular Floer data.
    Then we can choose a homotopy between $\mathcal{D}_0$ and $\mathcal{D}_1$ such that the continuation maps restrict to
    \begin{align*}
        \Psi_{\mathcal{D}_0,\mathcal{D}_1}: CF^{\lambda}(L_0,L_1; \mathcal{D}_0) &\longto CF^{\lambda + \osc(G-H)}(L_0,L_1; \mathcal{D}_1) \\
        \Psi_{\mathcal{D}_1,\mathcal{D}_0}: CF^{\lambda}(L_0,L_1; \mathcal{D}_1) &\longto CF^{\lambda + \osc(G-H)}(L_0,L_1; \mathcal{D}_0)
    \end{align*}
    for any $\lambda \in \Reals$ when taking into account the filtration.
\end{lemma} 
{
\begin{remark}
    Note again that all Hamiltonians are assumed to be compactly supported away from the boundary. 
    This is also true for all Hamiltonian isotopies here.
    If that was not the case, the above statement would not be true and 
    the continuation map would only be a quasi-isomorphism for specific Hamiltonians which grow linearly at infinity with certain slopes.
    The interested reader is referred to~\cite[Section 2.6]{gong-2023} for this setup.
\end{remark}}

Recall that in the setting of this lemma $\Psi_{\mathcal{D}_0,\mathcal{D}_1}\circ \Psi_{\mathcal{D}_1,\mathcal{D}_0}$ and $\Psi_{\mathcal{D}_1,\mathcal{D}_0}\circ \Psi_{\mathcal{D}_0,\mathcal{D}_1}$ are chain-isotopic to the identity. Thus, we obtain on homology that these two compositions are the same as the maps induced by the respective inclusions based on the action filtration.
These maps form an interleaving of the respective persistence modules.
By Theorem~\ref{thm:persistence_isometry} this implies a bound on the bottleneck distance, as follows:
\begin{corollary}
    Let $\mathcal{D}_0 = (G,J)$ and $\mathcal{D}_1 = (H,J')$ be two regular Floer data. Then the barcodes $\B(HF{(L_0,L_1; \mathcal{D}_0)}^\bullet)$ and $\B(HF{(L_0,L_1; \mathcal{D}_1)}^\bullet)$
    are $(\osc G-H)$-matched.
\end{corollary}
Next, we want to look at another chain map called the \textit{naturality map}.
We only give a brief description, for more information on the filtered naturality map we refer the reader to~\cite[Section 2.2.3]{biran-cornea-2022}.
For this let $\mathcal{D} = (H,J)$ be a regular Floer datum and $G$ another Hamiltonian.
Denote by $\phi^t_G$ the time-$t$ flow of the Hamiltonian vector field $X_G$.
For any two Hamiltonians $F_1,F_2: [0,1] \times M \to \Reals$ we denote by $F_1\sharp F_2: [0,1] \times M \to \Reals$ the map
\[F_1 \sharp F_2 (t,x) \coloneqq F_1(t,x) + F_2(t, {(\phi^{t}_{F_1})}^{-1}(x)).\]
Then there is a natural correspondence between $\Chords(L_0,L_1; H)$ and $\Chords(L_0,\phi_G^1(L_1); G \sharp H)$ given by the following map 
\begin{align*}
    \Chords(L_0,L_1; H) &\to \Chords(L_0,\phi_G^1(L_1); G\sharp H) \\
    y &\mapsto (t \mapsto \phi^t_G(y(t))).
\end{align*}
In order to upgrade this to a chain map, we introduce the \textit{push-forward Floer datum}
\[{(\phi_G)}_* \mathcal{D} = (G \sharp H, t \mapsto D\phi_G^t \circ J_t \circ D{(\phi^{t}_G)}^{-1}).\]
Note that the Hamiltonian chords $\Chords(L_0,\phi^1(L_1); G \sharp H)$ are generators of the Floer complex $CF(L_0,\phi_G^1(L_1); {(\phi_G)}_*\mathcal{D})$.
The chain map defined by linear extension of
\begin{align*}
    \mathcal{N}_G: CF(L_0,L_1; \mathcal{D}) &\to CF(L_0,\phi_G^1(L_1); {(\phi_G)}_*\mathcal{D}) \\
    y &\mapsto (t \mapsto \phi^t_G(y(t))),
\end{align*}
is called the \textit{naturality map}.
{This map is an isomorphism on the chain-level. For a proof see e.g.~\cite[Section 2.2.2]{leclercq-2008}. Thus, it also induces an isomorphism on homology: }
\[H(\mathcal{N}_G): HF(L_0,L_1; \mathcal{D}) \to HF(L_0,\phi_G^1(L_1); {(\phi_G)}_*\mathcal{D}).\]
Note that in order to define the filtration on the right-hand side, we have to make a specific choice of primitive for $\lambda$ on $\phi_G^1(L_1)$.
However, all such choices agree up to adding a constant since $L_1$ is connected.
{By choosing the primitive 
\[h_{\phi_G^1(L_1)}(\phi_G^1(x)) = h_{L_1}(x) + \int_0^1 \lambda(\partial_t\phi^t_G(x)) - G(t, \phi^t_G(x)) dt \]
for $\lambda$ on $\phi_G^1(L_1)$,
this map can be made action preserving.
Note that since all primitives agree up to addition by a constant, another choice of primitive merely shifts the filtration.
Thus, any results about the length of bars holds regardless of the chosen primitives.}

We can now combine naturality and continuation maps to obtain a useful result on Floer barcodes.
For this let $L_0,L_1 \subset M$ be exact Lagrangians as always and further assume that $L_0 \trans L_1$.
Let $H \in C^\infty([0,1]\times M, \Reals)$ be a Hamiltonian such that $\phi_H(L_1) \trans L_0$, where $\phi_H$ is the time-1 flow of $H$.
Then we can look at the following composition: 
\begin{center}
    \begin{tikzcd}
        {CF(L_0,L_1; 0)} \arrow[rr, "{\Psi_{0,\overline{H}}}"] &  & {CF(L_0,L_1; \overline{H})} \arrow[rr, "\mathcal{N}_H"] &  & {CF(L_0,\phi_H(L_1); 0)},
    \end{tikzcd}
\end{center}
where $\overline{H}(t,x) \coloneqq -H(t,\phi_H^t(x))$.
Here (and henceforth)
we drop the almost-complex structures from the notation for continuation maps for the sake of readability:
On homology this \textit{per se} just gives us the standard fact that Floer homology is invariant under Hamiltonian isotopies.
Now, by using Lemma~\ref{lem:continuation_action} we obtain the following filtered result:
for any $\lambda \in \Reals$ we have the following
\begin{center}
    \begin{tikzcd}
        {CF^\lambda(L_0,L_1; 0)} \arrow[r, "{\Psi_{0,\overline{H}}}"] & {CF^{\lambda+\osc(H)}(L_0,L_1; \overline{H})} \arrow[d, "\mathcal{N}_H", out = 0, in = 180] & \\
        & {CF^{\lambda+\osc(H)}(L_0,\phi_H(L_1); 0)} \arrow[d, "\mathcal{N}_{\overline{H}}", out = 0, in = 180] & \\
        & {CF^{\lambda+\osc(H)}(L_0,L_1; \overline{H})} \arrow[r, "{\Psi_{\overline{H},0}}"] & {CF^{\lambda+2\osc(H)}(L_0,L_1; 0)}.
    \end{tikzcd}
\end{center}
Since this composition induces the map on homology that is induced by the natural inclusion map, we again obtain an interleaving of persistence modules. To summarize:
\begin{lemma}\label{lem:continuation_interleaving} 
    Let $L_0 \trans L_1$ and let $H$ be a compactly supported Hamiltonian such that $\phi_H(L_1) \trans L_0$.
    Then the barcodes
    $\B(HF{(L_0,L_1; 0)}^\bullet)$ and $\B(HF{(L_0,\phi_H(L_1); 0)}^\bullet)$
    are $(\osc H)$-matched.
\end{lemma}
This implies the following 
Lipschitz properties of the boundary depth and the spectral invariants. This result is due to Leclercq for spectral invariants, see~\cite{leclercq-2008}, and Usher for the boundary
depth, see~\cite{usher-2013}.
{\begin{corollary}\label{cor:lipschitz}
    Let $L_0 \trans L_1$ and let $\phi$ be a compactly supported Hamiltonian diffeomorphism such that $\phi(L_1) \trans L_0$.
    Then \[\abs{\beta(HF{(L_0,L_1; 0)}^\bullet) - \beta(HF{(L_0,\phi(L_1); 0)}^\bullet)} \leq 2\cdot \norm{\phi}_{\Hof},\]
    where $\beta$ denotes the boundary depth of a persistence module.
    For any $\alpha \in HF(L_0,L_1)$ we have \[\abs{c(\alpha;HF{(L_0,L_1; 0)}^\bullet) -  c(\phi^*\alpha; HF{(L_0,\phi(L_1); 0)}^\bullet)} \leq \norm{\phi}_{\Hof},\]
    where $c$ denotes the spectral invariant of a class in the terminal vector space of a persistence module.
\end{corollary}}
\begin{remark}
    In the following sections, we will drop the superscript $\bullet$ from the notation for Floer theoretic spectral invariants and the boundary depth.
    This is as not to clutter the notation needlessly.
\end{remark}

\section{Setup}\label{sec:setup}
We will now define local models for the Hamiltonian diffeomorphisms and Dehn twists which we will later need for the proof of Theorem~\ref{thm:main_1}.
Henceforth, we will always consider the sphere $S^n$ endowed with its standard round metric.
{In this section we will only look at the local model for an $A_3$-configuration.
Recall that $L_0, L_1, L_2 \subset M$ are in an $A_3$-configuration if $L_0 \cap L_2 = \emptyset$,
and $L_1 \cap L_0$ and $L_1 \cap L_2$ each contain exactly one transverse intersection point.
Locally around $L_1$, this can be identified with 
the cotangent bundle of $S^n$ with two preferred fibers. 
These two fibers correspond to $L_0$ and $L_2$ near their intersection points with $L_1$.
The zero section of the cotangent bundle corresponds to $L_1$.
This is what we refer to as the \textit{local model}.}
We will assume throughout that $n \geq 2$.
The construction will translate verbatim to the neighborhood of an $A_3$-configuration of spheres in a Liouville domain $(M, \omega = d\lambda)$.

{We will now work with the local model.
We set $U_I \coloneqq \{\xi \in T^*S^n \vert \norm{\xi^\sharp} \in I\}$. Recall that we set $L_0 \coloneqq T_x^*S^n ,L_2 \coloneqq T_y^*S^n$, where 
$x,y \in S^n$ are two distinct points which are not antipodal.}
We also denote $L_1 \coloneqq o_{S^n} \subset D^*S^n$ and set $\delta \coloneqq d(x,y)$. 
{This slightly unconventional notation is used to emphasize that we are thinking of an $A_3$-configuration in the background.}
Choose ${0 < \hat{h}_0 < \check{h}_1 < \hat{h}_1 < \cdots < \check{h}_d < \hat{h}_d < \check{h}_{d+1} < 1}$.
Then set $\hbar \coloneqq \min_{i \in \{1,\dots,d\}} \abs{\check{h}_i - \hat{h}_i}$ and
choose $\iota \in \interval[open]{0}{\frac{\hbar}{2}}$.
Let $\theta \in C^\infty(\Reals,\interval[open right]{0}{\infty})$ be a smooth bump function such that:
\begin{enumerate}
    \item the support of $\theta$ is $\interval{\iota}{\hbar-\iota}$;
    \item $\theta(\frac{\hbar}{2}) = 2\pi$ is a global maximum;
    \item $\theta'\vert_{\interval[open]{\iota}{\frac{\hbar}{2}}} > 0$ and $\theta'\vert_{\interval[open]{\frac{\hbar}{2}}{\hbar-\iota}} < 0$;
    \item $\int_{\Reals} \theta(t) dt = 1$.
\end{enumerate} 
Then we define $\theta_v$ as follows for any $v \in \Reals^d$:
\[\theta_v(t) = \sum_{i=1}^d v_i \cdot\theta(t - \check{h}_i)\]
{Later on we will use the shorthand $\theta_i(\cdot) \coloneqq \theta(\cdot - \check{h}_i)$.}
The flow of the autonomous Hamiltonian 
\begin{align*}
    H_v: T^*S^n &\to \Reals  \\
    (x,\xi_x) &\mapsto \int_0^{\norm{\xi_x^\sharp}} \theta_v(t) dt,
\end{align*}
will be denoted by $\phi_v$. This defines a map $\Reals^d \to \Ham_c(T^*S^n)$ by $v \mapsto \phi_v$.
{
    We will now give an interpretation of this map in terms of the (co-)geodesic flow on $S^n$.
    For this let 
\[\sigma(t): T^*S^n\setminus o_{S^n} \to T^*S^n\setminus o_{S^n}\]
be the time-$t$
normalized cogeodesic flow.
This flow, closely related to the geodesic flow, is defined outside the zero section.
Recall, that the geodesic flow transports every cotangent vector by parallel transport along the geodesic 
associated to the tangent vector which corresponds to it under the musical isomorphism $\sharp: T^*S^n \to TS^n$.
Similarly, the flow $\sigma(t)$ transports cotangent vectors along the respective unit-speed geodesics for time $t$.
This action is of course only defined outside the zero section, which is preserved by the (unnormalized) geodesic flow.
Note that since all geodesics on the sphere are $2\pi$-periodic, the same is true for the normalized cogeodesic flow, i.e. $\sigma(t) = \sigma(t+2\pi)$ for any $t \in \Reals$.
Also note that $\sigma(\pi)$ is the antipodal map.
With these conventions, we have $\phi_v(\xi) = \sigma(\theta_v(\norm{\xi^\sharp}))(\xi)$ for any $\xi \not\in o_{S^n}$.
This means that outside the zero section, $\phi_v(\xi)$ is the cotangent vector that we obtain by transporting $\xi\not\in o_{S^n}$ along a geodesic with initial condition $\norm{\xi^\sharp}^{-1}\xi^\sharp$
for time $\theta_v(\norm{\xi^\sharp})$.}

After defining this map $\Reals^d \to \Ham_c(T^*S^n)$, we construct special representatives of the Dehn twist.
These will later be used in the way sketched in Section~\ref{sec:intro_proof}.
Let $\eps = \min_{i \in \{0,\dots,d\}} \abs{\check{h}_{i+1} - \hat{h}_i}$.
Let $\rho: \Reals \to \Reals$ be a monotone smooth function that is equal to $\pi$ on $\interval[open left]{-\infty}{\frac{1}{3}\eps}$ and equal to $0$ on $\interval[open right]{\frac{1}{2}\eps}{\infty}$.
We define for any $i \in \{0,\dots,d\}$ that 
\[\rho_i(t) \coloneqq \rho(t-\hat{h}_i).\]
Then we define the following models of a Dehn twist around $L_1$:
\begin{align*}
    \tau_i: D^*S^n &\to D^*S^n \\
    (x,\xi) &\mapsto \sigma(\rho_i(\norm{\xi^\sharp}))(x,\xi).
\end{align*}
{A priori, this map is only defined outside the zero section. However, we can extend it to the zero section. Recall that $\sigma(\pi)$ is the antipodal map. Since $\rho_i$ is equal to $\pi$ in a neighborhood of $0$, we can extend $\tau_i$ 
to the zero section via the antipodal map and the resulting map is still smooth and an exact symplectomorphism.
Note that this construction is essentially that given by Seidel in~\cite[Sec. 5a]{seidel-2000}.
}

By construction the supports of $\phi_v, \tau_0^2, \dots, \tau_d^2$ are pairwise disjoint for any $v \in \Reals^d$.
To be more precise, 
\begin{align*}
    \supp \phi_v &\subset U_{\interval[open]{\check{h}_1}{\hat{h}_1}} \cup \cdots \cup U_{\interval[open]{\check{h}_d}{\hat{h}_d}} \\
    \supp \tau_i^2 &\subset U_{\interval[open]{\hat{h}_i}{\check{h}_{i+1}}}
\end{align*}
for any $v \in \Reals^d$ and $i \in \{0,\dots,d\}$.
For an $A_3$-configuration $L_0,L_1,L_2 \subset M$ in an exact symplectic manifold $(M, d\lambda = \omega)$, there is a Weinstein neighborhood $U \subset M$ 
of $L_1$ which is symplectomorphic to {a disk cotangent bundle $D_r^*L_1$ for some small radius $r >0$}.
Further, by making $U$ smaller if necessary, we can assume that the intersections of $L_0$ and $L_2$ with $U$ correspond to fibers under this identification.
After rescaling, the construction above can be carried out in $U$.
We denote the obvious extensions of the maps constructed above by $\phi_v, \tau_0^2, \dots, \tau_d^2: M \to M$. 
Note that $v \mapsto \phi_v$ gives a group homomorphism $(\Reals^d,+) \to  (\Ham_c(M), \circ)$, which follows easily from the definition of the autonomous Hamiltonian $H_v$.

\begin{remark}
    {
    Note that the local model we have introduced here is only a computational device.
    The identification of a neighborhood of $L_1$ with a neighborhood of the zero section in $T^*S^n$
    is used in the definition of $\tau_0,\dots,\tau_d$ and $\phi_v$ in an essential way.
    These maps can be best understood in terms of the geodesic flow on the sphere, and the intersection points $\tau_i^{2k}(L_0) \cap \phi_v(L_2)$ we will 
    study later can be understood by looking at geodesics on $S^n$.
    This is what we use the local model for.
    However, we do not study the local model as a standalone setting and the Floer theory we use is always defined with respect to the appropriate ambient manifold, i.e. $M$ for Theorem~\ref{thm:main_1},
     $A_2^n$ for Corollary~\ref{cor:main_1} and $T^*S^n$ for Corollary~\ref{cor:main_2}.}
\end{remark}

\section{Action computation}\label{sec:action_comp}
In the following, we will restrict ourselves to the case $n \geq 2$\footnote{For $n=1$ the proof would have to be considerably adapted. However, this case is already implied by~\cite[Theorem 1.5]{zapolsky-2013}.}.
In this section fix $d \in \Nats$ and let $(M^{2n}, d\lambda=\omega)$ be a Liouville domain with $2c_1(M)=0$.
Further, we assume that $L_0,L_1,L_2 \subset M$ form an $A_3$-configuration of exact Lagrangian spheres.
We further want $L_1$ to be a framed sphere in the 
sense of~\cite[Sec. III.16a]{seidel-2008}.
Thus, we equip $L_1$ with an auxiliary choice of diffeomorphism $v: S^n \to L_1$ from the standard unit sphere in $\Reals^{n+1}$ to $L_1$.
We call $v$ the \textit{framing} of $L_1$.
{This diffeomorphism is an auxiliary choice, which is used in the proof later on. Our results about Hofer geometry are independent 
of this choice.}
Let $U$ be a neighborhood of $L_1$ which is symplectomorphic to a disk cotangent bundle $D_r^*L_1 \isom D_r^*S^n$.
The latter symplectomorphism is induced by the framing $v$.
By a slight perturbation of the framing if necessary, we can assume that $L_0 \cap U$ and $L_2 \cap U$ are identified locally 
with fibers over non-antipodal points.

Alternatively, we allow $M = T^*S^n$ with $L_0,L_1,L_2$ as in the last section. This is the same as only considering the 
local neighborhood $U$ discussed above.
{We use this unconventional notation to emphasize that the proof can be reused verbatim to obtain Corollary~\ref{cor:main_2}}.

Denote the tautological 1-form on $U$ by $\lambda_{\taut}$.
Since the first cohomology of $U$ vanishes,
we can write $\lambda$ in $U$ as $\lambda = \lambda_{\taut} + df$ where $f \in C^\infty(U,\Reals)$.
Assume $h_{L_0}: L_0 \to \Reals$ and $h_{L_2}: L_2 \to \Reals$ are primitives for $\lambda$ on $L_0$ and $L_2$.
These induce primitives for $\lambda$ on $\tau^{2k}_i(L_0)$ and $\phi_v(L_2)$. After possibly adding a constant to the primitives, we can assume $h_{L_0}$ agrees with $f$ on $L_0 \cap U$ and $h_{L_2}$ agrees with $f$ on $L_2 \cap U$.
We can see the induced primitives explicitly inside $U$.
A standard computation (see~\cite[Proposition 9.3.1]{mcduff-salamon-1998}) gives us the following:
\begin{align*}
     h_{\phi_v(L_2)}(\phi_v(\xi)) &= f(\xi) + \int_0^1 \lambda(\partial_t\phi^t_{v}(\xi)) - H_v(\phi^t_{v}(\xi)) dt \\
     &= f(\xi) + \int_0^1 \lambda_{\taut}(\partial_t\phi^t_{v}(\xi)) + df(\partial_t\phi^t_{v}(\xi)) dt - \int_0^1 H_v(\phi^t_{v}(\xi)) dt  \\
     &= \theta_v(\norm{\xi^\sharp})\norm{\xi^\sharp} +f(\phi_v(\xi)) - \int_0^{\norm{\xi^\sharp}}\theta_v(s)ds,
\end{align*}
where $\xi \in U \cap L_2$ is arbitrary. The norm used above is that induced by the round metric on $L_1$.
This is preserved by the (reparametrized) cogeodesic flow and thus also by $\phi_v$.
We thus obtain 
\begin{equation}\label{eq:primitive_phi}
    h_{\phi_v(L_2)}(\xi) = \theta_v(\norm{\xi^\sharp})\norm{\xi^\sharp} +f(\xi)   - \int_0^{\norm{\xi^\sharp}}\theta_v(s)ds,
\end{equation}
for any $\xi \in U \cap \phi_v(L_2)$.
Similarly, for all $\xi \in U \cap \tau^{2k}_i(L_0)$ we obtain
\begin{equation}\label{eq:primitive_tau}
    h_{\tau^{2k}_i(L_0)}(\xi) = 2k\rho_i(\norm{\xi^\sharp})\norm{\xi^\sharp} + f(\xi) -2k\int_0^{\norm{\xi^\sharp}} \rho_i(s)ds.
\end{equation}
Note that, while $\tau_i^{2k}$
is not a Hamiltonian diffeomorphism, the restriction to $M \setminus L_1$ is Hamiltonian.
This follows from the definition of $\tau_i$ given in Section~\ref{sec:setup}.
Thus,~\cite[Proposition 9.3.1]{mcduff-salamon-1998} can be applied away from $L_1$.
Since $\tau_i^2$ is the identity in a neighborhood of $L_1$, $f\vert_{L_0}$ is a valid primitive for $\lambda$ on ${\tau^{2k}_i(L_0)}$ near $L_1$.
It is easy to see from the above formula that $h_{\tau^{2k}_i(L_0)}$ and $f$ agree near $L_1$ (i.e.\ near the zero section in the local model), since $\rho_i(\norm{\xi^\sharp}) = \pi$ holds for any $\xi$
with $\norm{\xi^\sharp} \leq \frac{1}{3}\varepsilon$.
Thus, we have established $h_{\tau^{2k}_i(L_0)}$ as a primitive for $\lambda$ on $\tau^{2k}_i(L_0) \cap U$.

\section{Index computation}\label{sec:index_comp}
We now want to fix gradings on all relevant Lagrangians. 
Recall that $2c_1(M) = 0$, which means that $M$ allows Lagrangians to be $\Z$-graded in principle.
Since $H^1(S^n) = 0$ for $n \geq 2$, all spheres involved can be graded and
different choices for such gradings will differ at most by an integer shift.
Let $\widetilde{L}_0,\widetilde{L}_2$ be any gradings of $L_0,L_2$.
Possibly after an integer shift, we can arrange for these gradings to agree with the canonical gradings of $L_0 \cap U$ and $L_2 \cap U$ induced by the vertical distribution on $U$.
By~\cite{seidel-2000} there is a grading on $\tau_i$ such that $\widetilde{\tau_i}(\widetilde{L}_j) = \widetilde{L}_j[1-n]$ near $L_1$ and 
$\widetilde{\tau_i}(\widetilde{L}_j) = \widetilde{L}_j$ near $\partial U$ for $j = 0,2$.
For aesthetic reasons we choose a slightly less conventional grading: Namely, we grade $\tau_i$ such that $\widetilde{\tau_i}(\widetilde{L}_j) = \widetilde{L}_j$ near $L_1$ and 
$\widetilde{\tau_i}(\widetilde{L}_j) = \widetilde{L}_j[n-1]$ near $\partial U$ for $j = 0,2$.
Since this simply amounts to an integer shift, it is a harmless deviation from~\cite{seidel-2000}.
{Note that we can actually quantify \textit{near} here. Outside the support of $\tau_i$, the grading is simply shifted. To be more precise, the grading $\widetilde{\tau_i}(\widetilde{L}_j)$ agrees with the canonical grading 
$\widetilde{L}_j$ in the radial shell from $0$ to $ \hat{h}_i$.
In the radial shell from $\check{h}_{i+1}$ outwards, $\widetilde{\tau_i}(\widetilde{L}_j)$ agrees with $\widetilde{L}_j[n-1]$. This holds for $j=0,2$.}
For $\phi_v$ we use the canonical grading.
With these gradings in place, we can compute the degrees of the intersection points $\tau^{2k}_i(L_0) \cap \phi_v(L_2)$
as generators of the Floer complex $CF(\tau^{2k}_i(L_0), \phi_v(L_2))$.
{
\begin{lemma}\label{lem:degrees_phi}
    Assume $c \in \tau_i^{2k}(L_0) \cap \phi_v(L_2) \cap \supp \phi_v$. We set 
    \[m \coloneqq \left\lfloor \frac{\abs{\theta_v(\norm{c^\sharp})}}{\pi}\right\rfloor \in \Nats_0.\]
    If $\norm{c^\sharp} < \hat{h}_i$, we have
    \begin{align*}
        \mu(c; \widetilde{\tau}_i^{2k}(\widetilde{L}_0), \widetilde{\phi}_v(\widetilde{L}_2)) = \begin{cases}
            n + (n-1)m                  & \text{if } \theta_v(\norm{c^\sharp}) > 0, \theta'_v(\norm{c^\sharp}) > 0; \\
            n + (n-1)m - 1              & \text{if } \theta_v(\norm{c^\sharp}) > 0, \theta'_v(\norm{c^\sharp}) < 0; \\
            - (n-1)m  + 1               & \text{if } \theta_v(\norm{c^\sharp}) < 0, \theta'_v(\norm{c^\sharp}) > 0; \\
            - (n-1)m                    & \text{if } \theta_v(\norm{c^\sharp}) < 0, \theta'_v(\norm{c^\sharp}) < 0. 
        \end{cases} 
    \end{align*}
    If $\norm{c^\sharp} > \check{h}_{i+1}$ we have
    \begin{align*}
        \mu(c; \widetilde{\tau}_i^{2k}(\widetilde{L}_0), \widetilde{\phi}_v(\widetilde{L}_2)) = \begin{cases}
            2k(n-1) + n + (n-1)m                  & \text{if } \theta_v(\norm{c^\sharp}) > 0, \theta'_v(\norm{c^\sharp}) > 0; \\
            2k(n-1) + n + (n-1)m - 1              & \text{if } \theta_v(\norm{c^\sharp}) > 0, \theta'_v(\norm{c^\sharp}) < 0; \\
            2k(n-1) - (n-1)m  + 1                 & \text{if } \theta_v(\norm{c^\sharp}) < 0, \theta'_v(\norm{c^\sharp}) > 0; \\
            2k(n-1) - (n-1)m                      & \text{if } \theta_v(\norm{c^\sharp}) < 0, \theta'_v(\norm{c^\sharp}) < 0, 
        \end{cases} 
    \end{align*}
    where we identify $c$ with the respective covector under the identification of $U$ with a neighborhood of the zero section in $T^*S^n$. 
\end{lemma}}
\begin{proof}
    First we consider the case where $\norm{c^\sharp} < \hat{h}_i$ holds. {This radial band lies outside the support of $\tau_i^2$, and, by our earlier discussion, $\widetilde{\tau}_i$ has no effect on the grading of $L_0$ at $c$.}
    Then
     \[\mu(c; \widetilde{\tau}_i^{2k}(\widetilde{L}_0), \widetilde{\phi}_v(\widetilde{L}_2)) = \mu(c; \widetilde{L}_0, \widetilde{\phi}_v(\widetilde{L}_2)) .\]
    Thus, it merely remains to compute $\mu(c; \widetilde{L}_0, \widetilde{\phi}_v(\widetilde{L}_2))$.
    {
    We consider the paths 
     $\lambda_0 \equiv \Tvert_c$ and $\lambda_1(t) = {(D\phi^t_{v}{(\Tvert)})}_c$ of Lagrangian subspaces of $T_c T^*S^n$.
    Here $\Tvert \subset TT^*S^n$ denotes the vertical subbundle.
    Note that these paths can be seen as projections of paths 
    $\tilde\lambda_0,\tilde\lambda_1: \interval{0}{1} \to \LG^\infty$ with 
    $\tilde\lambda_0(0) \equiv \widetilde{L}_0(c)$, $\tilde\lambda_1(0) =\tilde{L}_0(c)$ and $\tilde{\lambda}_1(1) =\widetilde{\phi}_v(\widetilde{L}_2)(c)$.
    This follows from the fact that the canonical grading of $L_0$ projects down to the vertical subbundle by definition.
    Further, the grading of $\phi_v$ is defined by lifting $D\phi_v$ to the universal cover. Thus, we can lift $\lambda_1$ to $\tilde{\lambda}_1$, and we can use $\lambda_0$ and $\lambda_1$ to compute the index of $c$ as described in~\eqref{eq:abs_grading}.}
    We have
    \[ \mu(c; \widetilde{L}_0, \widetilde{\phi}_v(\widetilde{L}_2))= \frac{n}{2}-\mu(\lambda_0,\lambda_1) ,\]
    where $\mu(\lambda_0,\lambda_1)$ is the Maslov index of these two paths.
    The relationship between $\mu(\lambda_0,\lambda_1)$ in this setup and the conjugacy index (i.e.\ multiplicity) of conjugate points along the geodesic $\gamma$ on $S^n$ connecting $x$ and $y$ is standard in the literature for the genuine geodesic flow.
    The essential argument is originally due to Duistermaat in~\cite{duistermaat-1976} and a proof in the modern language can be found in~\cite[Prop. 6.38]{robbin-salamon-1995}.
    
    {
    The argument below essentially follows~\cite[Prop. 6.38]{robbin-salamon-1995} with some adjustments made to account for the reparametrization of the geodesic flow. 
    We include it here for the sake of completeness and because the results of this paper depend on these indices in an essential way.

    In the following we denote the cogeodesic flow by $G(t): T^*S^n \to T^*S^n$ for the time-$t$ cogeodesic flow.
    We will later use the following presentation of the differential of the cogeodesic flow, see e.g.~\cite[Lemma 1.40]{paternain-1999},
    \begin{align*}
        DG(t): T_\xi T^*S^n &\to T_{G(t)(\xi)}T^*S^n \\
        \eta &\mapsto (J_{\eta}(t), \dot J_{\eta}^\flat(t)),
    \end{align*}
    where $\xi \in T^*S^n$ and $J_{\eta}$ is the unique Jacobi field along the geodesic $t \mapsto pr(G(t)(\xi))$ with $J(0) = \eta^{hor}$ and $\dot J(0) = (\eta^{vert})^\sharp$.
    Here $pr: T^*S^n \to S^n$ is the bundle map, and $\eta^{hor}$ and $\eta^{vert}$ denote the horizontal and vertical part of the covector $\eta \in T_\xi T^*S^n$.
    As before, we denote the musical isomorphisms associated to the standard round metric by $\flat: TS^n \to T^*S^n$ and $\sharp: T^*S^n \to TS^n$.
    We recall another standard fact from Riemannian geometry, namely
    that on a round sphere $S^n$ two points are conjugate along a great circle if they are antipodal.
    The conjugacy index of antipodal points along any great circle is $n-1$.
    The latter is a well-known fact in Riemannian geometry, see e.g.~\cite[p. 299]{lee-2018}. 

    By definition, we have $\phi_v^t(\xi) = G(t\cdot \theta_v(\norm{\xi^\sharp}) \cdot \norm{\xi^\sharp}^{-1})(\xi)$ for $\xi \in T^*S^n \setminus o_{S^n}$.
    Since $\norm{\xi^\sharp} = 0$ if $\xi \in o_{S^n}$ we extend this definition, by a slight abuse of notation, to $T^*S^n$.
    Note that the cogeodesic flow can be seen as the flow of the cogeodesic vector field $V_G(\xi) = (\xi^\sharp, 0)$, 
    where this notation indicates that the vertical part of $V_G(\xi)$ vanishes and the horizontal part is $\xi^\sharp$.
    In particular, $\frac{d}{dt}G(t)(\xi) = V_G(G(t)(\xi))$ for any $t \in \Reals, \xi \in T^*S^n$.
    We can then directly compute, using the gradient of the norm and the chain rule, that 
    \begin{equation}\label{eq:Dphi}
        \begin{split} 
        D\phi^t_v(\xi)(\eta) &= DG(t\cdot \theta_v(\norm{\xi^\sharp}) \cdot \norm{\xi^\sharp}^{-1})(\xi)(\eta) \\
        & \qquad+ t\cdot \frac{\langle \xi^\sharp,(\eta^{vert})^\sharp \rangle_{S^n}}{\norm{\xi^\sharp}^2}\cdot\left[  \theta'_v(\norm{\xi^\sharp}) - \frac{\theta_v(\norm{\xi^\sharp})}{\norm{\xi^\sharp}} \right] V_G(\phi^t_v(\xi)),
        \end{split}
    \end{equation}
    where $\eta \in T_\xi T^*S^n$ and $\langle \cdot, \cdot \rangle_{S^n}$ is the round metric on the sphere and $\eta^{vert}$ is the projection of $\eta$ to the vertical subbundle.
    In the following, we use $\alpha = \theta_v(\norm{c^\sharp}) \cdot \norm{c^\sharp}^{-1}$ as shorthand.

    In~\eqref{eq:Dphi}, we notice immediately that the second term vanishes whenever $\eta^{vert} \perp \xi$.
    Let $\gamma: \Reals \to S^n$ be the geodesic uniquely defined by $\dot\gamma(\alpha) = c^\sharp$ and $\gamma(\alpha) = x$.
    By the formula for $\phi_v$ in terms of the cogeodesic flow given above, this implies that $\gamma(0) = y$.
    Let $e_1, \dots, e_n: \Reals \to TT^*S^n$ be a moving parallel orthogonal frame along $\gamma$ such that $e_1( \alpha
    ) = c^\sharp$.
    In particular, $\linspan\{(0,e^{\flat}_1(t)), \dots, (0,e^{\flat}_n(t))\}$ is the vertical subspace over $\dot\gamma^\flat(t)$.
    Note that we can trivialize the rank $2n$ symplectic bundle $\dot\gamma^* TT^*S^n$ in such a way that all vertical subspaces
    are identified with $0 \times \Reals^{n} \subset \Reals^{2n}$.
    We fix such a trivialization for the remainder of the argument, in order to apply the machinery from~\cite{robbin-salamon-1993}
    to compute $\mu(\hat{\lambda}_0,\hat{\lambda}_1)$.

    It becomes more convenient to work with two slightly more geometrically meaningful paths than $\lambda_0,\lambda_1$.
    Namely, set 
    \begin{align*}
        \hat{\lambda}_0(t) &= D\phi_v^{-t}(\lambda_0(t)) = D\phi_v^{-t}(\Tvert_c) \\
        \hat{\lambda}_1(t) &= D\phi_v^{-t}(\lambda_1(t)) = \Tvert_{\phi_v^{-t}(c)}.
    \end{align*}
    By the naturality property of the Maslov index (cf.~\cite[Theorem 3.1]{robbin-salamon-1993}), 
    we have that $\mu(\lambda_0,\lambda_1) = \mu(\hat{\lambda}_0,\hat{\lambda}_1)$.
    Recall that $(0,e_1^\flat(\alpha)),\dots,(0,e_n^\flat(\alpha))$ span $\Tvert_c$.
    Thus, 
    \[\hat{\lambda}_0(t) = \linspan \left\{ D\phi_v^{-t}((0,e_1^\flat(\alpha)),\dots,D\phi_v^{-t}((0,e_n^\flat(\alpha)))  \right\},\]
    we denote this basis by $g_i(t) = D\phi_v^{-t}(0,e_i^\flat(\alpha))$ for $i \in \{1,\dots,n\}$.
    Now, notice that the Jacobi field $J_{(0,c)}$ is the variational vector field of the family $s \mapsto \gamma(\alpha+e^s\cdot t)$ of reparametrizations of $\gamma$.
    Thus, $J_{(0,c)}(t) = (t\cdot \dot\gamma(\alpha+t), \dot\gamma(\alpha+t))$ for all $t \in \Reals$.
    With this in mind, we can compute that 
    \begin{align*}
     g_1(t) &= D\phi_v^{-t}(0,e_1^\flat(\alpha)) \\
        &= \left(- t\alpha \cdot \dot\gamma(\alpha(1-t)),\dot\gamma(\alpha(1-t))\right) - t\cdot \left[  \theta'_v(\norm{c^\sharp}) - \alpha \right] V_G(\phi^{-t}_v(c)) \\
        &= \left(- t\alpha \cdot e_1(\alpha(1-t)),e_1^\flat(\alpha(1-t))\right) - t\cdot \left[  \theta'_v(\norm{c^\sharp}) - \alpha \right] e_1^\flat(\alpha(1-t)) \\
        &= (-t\cdot \theta'_v(\norm{c^\sharp})\cdot e_1(\alpha(1-t)), e^\flat_1(\alpha(1-t)) ),
    \end{align*}
    where we use the $(\cdot,\cdot)$ notation to denote the horizontal and vertical part of a tangent vector in $TT^*S^n$.
    We also use the fact that $e_1(\alpha(1-t))$ is the parallel transport of $c^\sharp$ along $\gamma$ from time $\alpha$ to time $\alpha(1-t)$.
    This follows directly from the fact that $e_1,\dots,e_n$ form a parallel frame and that $e_1(\alpha) = c^\sharp$.
    By the assumption of transversality, $\theta'_v(\norm{\xi}) \not= 0$ for any intersection point $\xi$ and thus for $c$.
    Thus, we can conclude that the horizontal part of $g_1$ never vanishes for $t>0$.
    We can therefore ignore it in our analysis of crossings for $t>0$.

    Since the frame $e_1, \dots, e_n$ is always orthogonal, we can see directly that the second line in~\eqref{eq:Dphi} vanishes in $g_i(t)$ for $i \geq 2$.
    In particular, by~\eqref{eq:Dphi}, 
    \begin{align*}
        g_i(t) &= D\phi_v^{-t}(0,e_i^\flat(\alpha)) \\
        &= (J_{(0,e_i^\flat(\alpha))}(-\alpha\cdot t), \dot J^\sharp_{(0,e_i^\flat(\alpha))}(-\alpha\cdot t)),
    \end{align*}
    for $i \in \{2,\dots,n\}$.

    Assume that $0 < t_0 < 1$ is a crossing, i.e.\ that 
    \[\dim_{\Reals} \lambda_0(t_0) \cap \lambda_1(t_0) = \dim_{\Reals} \hat{\lambda}_0(t_0) \cap \hat{\lambda}_1(t_0)> 0.\]
    We can see directly that $\dim_{\Reals} \lambda_0(t_0) \cap \lambda_1(t_0)$ is the conjugacy index of the points $\gamma(\alpha(1-t_0))$ and $\gamma(\alpha) = y$
    along the geodesic $\gamma$. This follows from~\eqref{eq:Dphi} and the decomposition of the differential of the cogeodesic flow in terms of Jacobi fields.
    Thus, $ \dim_{\Reals} \hat{\lambda}_0(t_0) \cap \hat{\lambda}_1(t_0)$ is the number of linearly independent Jacobi fields for $\gamma$ that vanish at $\gamma(\alpha(1-t_0))$ and $\gamma(\alpha)$, i.e.\
    the conjugacy index of these points.
    Since we are working on a round sphere, this implies $\dim_{\Reals} \hat{\lambda}_0(t_0) \cap \hat{\lambda}_1(t_0) = n-1$.
    In particular this implies that $\hat{\lambda}_0(t_0) \cap \hat{\lambda}_1(t_0) = \linspan \{g_2(t_0), \dots, g_n(t_0)\}$.

    The crossing index of this crossing is the signature of the crossing form $\Gamma(\lambda_0,\lambda_1,t_0)$, which is defined by 
    \begin{align*}
        \Gamma(\lambda_0,\lambda_1,t_0)g_i(t_0) &= \frac{d}{ds}\Big\vert_{s=t_0} \omega((g_i^{hor}(t_0),g_i^{vert}(t_0)),(-g_i^{hor}(s),0)) 
        \\ &= \left\langle \frac{d}{ds}\Big\vert_{s=t_0}g_i^{hor}(s),g_i^{vert}(t_0)\right\rangle_{S^n},
    \end{align*}
    where $i \in \{2,\dots,n\}$ and 
    $g_i^{vert}$ and $g_i^{hor}$ denote the vertical and horizontal parts of the respective covectors.
    Here we implicitly use the trivialization of $\dot\gamma^*TT^*S^n$ we chose earlier to define the crossing form.
    We refer the reader to~\cite{robbin-salamon-1993} and~\cite{robbin-salamon-1995} for the definition and elementary properties of the crossing form.
    Given our previous considerations, we 
    have 
    \begin{align*}
        \left\langle \frac{d}{ds}\Big\vert_{s=t_0}g_i^{hor}(s),g_i^{vert}(t_0)\right\rangle_{S^n}
         &= 
        \left\langle \frac{d}{ds}\Big\vert_{s=t_0} J_{(0,e_i^\flat(\alpha))}(-\alpha\cdot s), \dot J{(0,e_i^\flat(\alpha))}(-\alpha\cdot t_0)\right\rangle_{S^n}
        \\ &= -\frac{\theta_v(\norm{c^\sharp})}{\norm{c^\sharp}} \norm*{\dot J_{(0,e^\flat_i(\alpha))}(-\alpha t_0)}^2 \not= 0,
    \end{align*}  
    Thus, the signature of $\Gamma(\lambda_0,\lambda_1,t_0)$ is $-\sign \theta_v(\norm{c^\sharp}) \cdot (n-1)$.
    
    Note that $\dim_{\Reals} \hat{\lambda}_0(0) \cap \hat{\lambda}_1(0) = n$.
    The computation of $\Gamma(\lambda_0,\lambda_1,t_0)g_i(t_0)$ performed for $i =2,\dots,n$ in the $t>0$ case still remains valid in the $t=0$ case.
    However, we now must also compute $\Gamma(\lambda_0,\lambda_1,0)g_1(0)$.
    We obtain that 
    \begin{align*}
        \Gamma(\lambda_0,\lambda_1,0)g_1(0) &= \frac{d}{ds}\Big\vert_{s=0} \omega((0,c^\sharp),(-g_1^{hor}(s),0)) \\
        &= \frac{d}{ds}\Big\vert_{s=0} -s\cdot \theta'_v(\norm{c^\sharp})\cdot \norm{c^\sharp}^2
        =  -\theta'_v(\norm{c^\sharp})\cdot \norm{c^\sharp}^2 \not= 0, 
    \end{align*}
    and thus the crossing at $t=0$ has index $-(\sign \theta'_v(\norm{c^\sharp}) + \sign \theta_v(\norm{c^\sharp}) \cdot (n-1))$.
    Again we implicitly use the trivialization of $\dot\gamma^*TT^*S^n$ to define the crossing form.
    The Maslov index $\mu(\hat{\lambda}_0,\hat{\lambda}_1)$ is obtained as a sum over the crossing indices of all crossings. 
    The crossings at $t=0$ and $t=1$ are weighted by $\frac{1}{2}$, see~\cite{robbin-salamon-1993} for details.
    In our case there is no crossing at $t=1$ by the assumption of transversality.
    Since we have seen that crossings for $t>0$ corresponds to conjugate points, it only remains to count those.
    Recall, that on a round sphere $S^n$ two points are conjugate along a great circle exactly when they are antipodal.
    Thus, the number of conjugate points along the relevant segment of $\gamma$ is given by $\left\lfloor \frac{\abs{\theta_v(\norm{c^\sharp})}}{\pi}\right\rfloor$.
    Overall, we therefore obtain the formula
    \begin{align*}
        \mu(\hat{\lambda}_0,\hat{\lambda}_1) &= -\frac{\sign \theta_v'(\norm{c^\sharp}) + \sign(\theta_v(\norm{c^\sharp}))(n-1)}{2} \\ &\qquad-  \sign(\theta_v(\norm{c^\sharp}))(n-1)\left\lfloor \frac{\abs{\theta_v(\norm{c^\sharp})}}{\pi}\right\rfloor.
    \end{align*} }
    Since
    \begin{align*}
        \mu(c;\widetilde{\tau}_i^{2k}(\widetilde{L}_0), \widetilde{\phi}_v(\widetilde{L}_2)) &= \frac{n}{2}-\mu(\lambda_0,\lambda_1) = \frac{n}{2}-\mu(\hat{\lambda}_0,\hat{\lambda}_1)
    \end{align*}
    this implies the relevant cases in the lemma.

    Note that essentially the same argument applies for the case when $ \norm{c^\sharp} > \check{h}_{i+1}$. 
    Recall from above that $\widetilde{\tau}_i$ acts as a simple shift on $\widetilde{L}_0$ in this radial shell. This shift is all that has to be accounted for.
    We obtain 
    \begin{align*} \mu(c; \widetilde{\tau}_i^{2k}(\widetilde{L}_0), \widetilde{\phi}_v(\widetilde{L}_2))
        &= \mu(c; \widetilde{L}_0[2k(n-1)], \widetilde{\phi}_v(\widetilde{L}_2)) = \mu(c; \widetilde{L}_0, \widetilde{\phi}_v(\widetilde{L}_2)) + 2k(n-1),
    \end{align*}
    thus completing the proof.
\end{proof}
The remaining degrees obey the standard behavior for the Dehn twist (up to our non-standard shift), which we restate for the reader's convenience.
{Recall that $\delta = d_{S^n}(x,y)$ is the distance between the unique intersection points in $L_0 \cap L_1$ and in $L_1 \cap L_2$ on $L_1$.
The distance is taken with respect to the round metric under the identification of $L_1$ with $S^n$ given by the framing $v$.}
{
\begin{lemma}\label{lem:degrees_tau} 
    Assume $\xi \in \tau_i^{2k}(L_0) \cap \phi_v(L_2)$ lies outside $\supp \phi_v$. Then we have 
    \[\mu(\xi; \widetilde{\tau}_i^{2k}(\widetilde{L}_0), \widetilde{\phi}_v(\widetilde{L}_2)) = n + (2 k - 1 - m) (n - 1),\]
    where $m$ is such that {$2k\rho_i(\norm{\xi^\sharp}) = \pi \cdot m + \delta$ or $2k\rho_i(\norm{\xi^\sharp}) = \pi \cdot (m + 1) - \delta$.} 
    Since $\rho_i(\Reals) = \interval{0}{\pi}$, this implies that $m$ lies between $0$ and $2k-1$.
\end{lemma}
\begin{proof}
    This proof is almost the same as before.
    First note that, by construction of $\tau_i$, the fact that $\xi$ is an intersection point implies that $2k\rho_i(\norm{\xi^\sharp}) = \pi \cdot m + \delta$ or $2k\rho_i(\norm{\xi^\sharp}) = \pi \cdot (m + 1) - \delta$ for some $m \in \Nats_0 \cap \interval{0}{2k-1}$.
    Since $\xi$ lies outside the support of $\phi_v$ and $\phi_v$ is a Hamiltonian diffeomorphism and $\widetilde{\phi}_v$ is the canonical grading,
    \[\mu(\xi; \widetilde{\tau}_i^{2k}(\widetilde{L}_0), \widetilde{\phi}_v(\widetilde{L}_2)) 
        = \mu(\xi; \widetilde{\tau}_i^{2k}(\widetilde{L}_0), \widetilde{L}_2) .\]
    {Thus, it remains to compute $\mu(\xi; \widetilde{\tau}_i^{2k}(\widetilde{L}_0), \widetilde{L}_2)$.  
    This can be done in exactly the same fashion as in the proof of Lemma~\ref{lem:degrees_phi}.
    To see this, note that $\xi \not\in o_{S^n}$.
    Thus, in the relevant part of $T^*S^n$, we can write $\tau_i$ as a reparametrization of the cogeodesic flow.
    To be precise,
    \[\tau_i^{2k}(\xi) = G(2k\rho_i(\norm{\xi^\sharp}) \cdot \norm{\xi^\sharp}^{-1})(\xi),\]
    which follows directly from the construction of our model Dehn twists in Section~\ref{sec:setup}.
    We introduce the auxiliary flow 
    \[\psi^t(\xi) = G(t\cdot2k\rho_i(\norm{\xi^\sharp}) \cdot \norm{\xi^\sharp}^{-1})(\xi),\]
    for $\xi \in T^*S^n \setminus o_{S^n}$ and $t \in \Reals$.
    Note that outside the zero section $\psi^1 = \tau_i^{2k}$.
    Then we define $\lambda_1 \equiv \Tvert_\xi$ and $\lambda_0(t) = D\psi^t(\Tvert)_c$.
    By lifting these paths to the universal cover, we obtain paths which can be used to compute the grading as described in Section~\ref{sec:floer_theory_setup}
    and as done in the proof of Lemma~\ref{lem:degrees_phi}.
    This is exactly the same setup as in the proof of Lemma~\ref{lem:degrees_phi} with the roles of $\lambda_0$ and $\lambda_1$ reversed.
    Therefore, our previous argument applies verbatim. Therefore, we obtain the same formula as at the end of the proof of Lemma~\ref{lem:degrees_phi}, except for the sign which has to be reversed:
    \begin{align*} \mu(\lambda_0,\lambda_1) &= \frac{\sign \rho_i'(\norm{\xi^\sharp}) + \sign(\rho_i(\norm{\xi^\sharp}))(n-1)}{2}  \\ &\qquad + \sign(\rho_i(\norm{\xi^\sharp}))\left\lfloor \frac{2k\rho_i(\norm{\xi^\sharp})}{\pi} \right\rfloor(n-1) \\
    &= \frac{n-2}{2} + (n-1)m.  
    \end{align*}

    However, note that we use a non-standard grading for $\tau$.
    We have so far not incorporated this choice into the computation.
    The flow $\Psi$ is the identity near the boundary of $U$, i.e. outwards of the outermost radial shell.
    Thus, it acts trivially on the grading of the Lagrangians in this region.
    However, our chosen grading $\widetilde{\tau}$ acts as a $2k(n-1)$-shift near the boundary of $U$.
    Thus, we need to apply this shift to the resulting index.
    Here we use the uniqueness of a grading for Dehn twists up to a shift, see~\cite{seidel-2000}.
    This implies
    \begin{align*}
        \mu(\xi;\widetilde{\tau}_i^{2k}(\widetilde{L}_0), \widetilde{\phi}_v(\widetilde{L}_2)) &= 2k(n-1) + \frac{n}{2} - \left(\frac{n-2}{2} + (n-1)m\right) \\ &= 1 + (2 k - m) (n - 1)
        \\ &= n + (2 k - 1 - m) (n - 1).
    \end{align*}} 
\end{proof}}
{
\begin{remark}
    Note that these degrees differ from the results obtained using the more conventional grading given in~\cite{seidel-2000} for the Dehn twist by a $2k(n-1)$-shift.
    This is due to the non-canonical grading we impose on $\tau^2_i$, which acts as a non-trivial shift near the boundary.
    This is done to simplify later computations.
    The degrees given in~\cite{frauenfelder-schlenk-2005} also differ from the canonical grading by yet a different shift.
    However, up to this $\Z$-ambiguity, all of these gradings do of course agree.
\end{remark}
\begin{remark}\label{rmk:intersection_point_count}
    Note that in the proof of Lemma~\ref{lem:degrees_tau} we have also implicitly seen the 
    number of intersection points in $\tau_i^{2k}(L_0) \cap \phi_v(L_2) \cap \supp \tau_i^{2k} = \tau_i^{2k}(L_0) \cap L_2$.
    Since $\supp \tau_i^{2k}$ lies away from $L_1$ (i.e.\ the zero section in the local model), 
    we can use $\tau_i^{2k}(\xi) = \sigma(2k\cdot \rho_i(\norm{\xi^\sharp})(\xi)$.
    Let $\gamma: \Reals \to S^n$ be the unit-speed geodesic with $\gamma(0) = x$ and $\gamma(\delta) = y$.
    Note that up to reparametrization, this is the only geodesic connecting $x$ to $y$ since $x,y$ are two distinct non-antipodal points on a round sphere. 
    Recall that $\delta = d_{S^n}(x,y)$.
    Now assume that we have $\xi \in T_x^*S^n$ such that $\tau_i^{2k}(\xi) \in L_2$.
    Since $\sigma$ is the cogeodesic flow, $\xi$ has to be transported to some covector over $y$ along a geodesic. This implies $\xi^\sharp = s \cdot \dot\gamma(0)$.
    We can now solve for $s$ and obtain that it must satisfy $2k\rho_i(\abs{s}) = 2\pi \cdot m \pm \delta$ for some $m \in \Nats_0$.
    Solutions with $+\delta$ are realized geometrically for $s > 0$ and solutions with $-\delta$ are realized geometrically for $s < 0$.
    Since the range of $\rho_i$ is $\interval{0}{\pi}$, there is a solution for any $m \in \{0,\dots,k\}$.
    Since $\rho_i$ is monotone, there is exactly one solution (to the equation $2k\rho_i(s) = \delta$) for $m=0$, exactly two solutions (one for each sign) for $m \in \{1,\dots,k-1\}$,
    and exactly one solution (to the equation $2k\rho_i(s) = 2\pi k - \delta$) for $m = k$.
    Thus, in total there are exactly $2k$ intersection points in $\tau^{2k}(L_0) \cap \phi_v(L_2) \cap \supp \tau_i^{2k}$ which all live in different degrees by Lemma~\ref{lem:degrees_tau} since $n \geq 2$.
\end{remark}
}
\begin{remark}\label{rmk:fs}
	Frauenfelder and Schlenk prove in~\cite[Theorem 2.13]{frauenfelder-schlenk-2005}
	that $HF(\tau_i^{2k}(L_0),L_2)$ has rank $2k$ for the local model, i.e. two non-antipodal fibers in a sphere cotangent bundle.
	Their proof carries over to an $A_3$-configuration in $M$ for trivial reasons whenever $n\geq 3$, since the Floer boundary map has to vanish for 
	degree reasons when $n \geq 3$.
	Since we do not work with $n=1$, it remains to verify that the $n=2$ case carries over as well.
	Here we can use the fact that, by a standard maximum principle argument, finite-energy $J$-holomorphic strips connecting two intersection points in $U$
	must stay in $U$. See e.g.~\cite[Lemma 2.6]{frauenfelder-schlenk-2005}.
    Here we use that $U$ is symplectically convex.
    We can see this by using the fact that the (locally defined) Liouville vector field with respect to $\lambda_{\taut}$ is inwards-pointing on the boundary of $U$ (note our sign convention here). 
    Given that this Liouville vector field also lies in the tangent space of the Lagrangians at the intersection points with the boundary of $U$, 
    the maximum principle argument applies.
    We refer the reader to~\cite{frauenfelder-schlenk-2005} for all further details.
    The same argument applies in the $A_2$-plumbing.
	Thus, $\dim_{\field_2}HF(\tau_i^{2k}(L_0),L_2) = 2k$ in all relevant cases.
\end{remark}
\begin{remark}\label{rmk:gradings_bullet}
    Note that by construction of $\theta_v$, we have $\abs{\theta_v}(\Reals) \subset \interval{0}{\norm{v}_\infty \cdot 2\pi}$.
    Thus, for any $t \in \Reals$ we get
    \[\left\lfloor \frac{\abs{\theta_v(t)}}{\pi}\right\rfloor \leq \lfloor 2 \norm{v}_\infty  \rfloor \leq 2 \norm{v}_\infty.\]
    Now, let us recap which degrees might have generators in them for some $v \in \Reals^d$ and $k \in \Nats$.
    Due to the way we carefully constructed $\phi_v$ and $\tau_0,\dots,\tau_d$, we have \[\tau^{2k}_i(L_0) \cap \phi_v(L_2) = L_0 \cap \phi_v(L_2) \amalg\tau^{2k}_i(L_0) \cap L_2.\]
    All of these generators lie in $U$, and we use the identification of $U$ with a neighborhood of the zero section in $T^*S^n$ when we talk about radial bands below.
    In each of these parts, we can describe the possible degrees of the generators:
    \begin{enumerate}
        \item By Lemma~\ref{lem:degrees_phi}, generators which lie in a radial band from $0$ to $\hat{h}_i$ can be at most one degree away from the following degrees:
        \begin{center}
            \begin{tabular}{ c|c|c|c } 
                $-\floor{2 \norm{v}_{\infty}}  (n-1)$ & $-(\floor{2\norm{v}_\infty}-1) (n-1)$ & \dots & $0$ \\
                $n$  & $n + (n-1)$ & \dots & $n + \floor{2 \norm{v}_{\infty}}(n-1)$
            \end{tabular}
        \end{center}
        Here, in the first row the actual degree might be higher by one and in the second row it might be lower by one.
        {Note that the columns are not meaningful here and each row lists a different progression of degrees. The first row is associated with positive values of some $v_i$ and the second row with negative values.}
        \item  Generators which lie in the radial band from $\hat{h}_i$ to $\check{h}_{i+1}$ can appear only in the following degrees according to Lemma~\ref{lem:degrees_tau}:
        \begin{center}{
            \begin{tabular}{ c|c|c|c|c } 
                $n$ & $n + (n-1)$ & $n + 2(n-1)$ &\dots & $n + (2k-1) (n-1)$ 
            \end{tabular}
            }
        \end{center}
        \item By the other case of Lemma~\ref{lem:degrees_phi}, generators with a norm higher or equal to $\check{h}_{i+1}$ can be at most one degree away from the following degrees:
        \begin{center}
            \begin{tabular}{ c|c|c } 
                $2k(n-1) - \floor{2 \norm{v}_{\infty}}(n-1)$ & \dots  & $2k(n-1)$ \\
                $2k(n-1) + n$  & \dots & $2k(n-1) + n + \floor{2 \norm{v}_{\infty}}(n-1)$
            \end{tabular}
        \end{center}
        Again, the actual degree might be higher by one in the first row and lower by one in the second row.
    \end{enumerate}
    {
        {For us, the essential piece of information is this}: Any generator $c \in \tau_i^{2k}(L_0) \cap \phi_v(L_2)$ that lies in $\supp \phi_v$ --- and thus outside the support of any $\tau_i^2$ --- is covered by the first or third case above.
        If case (1) above applies,
        $c$ has a degree below or equal to $n + \floor{2 \norm{v}_{\infty}}(n-1)$.  If case (3) applies, it has a degree above or equal to $(2k - \floor{2 \norm{v}_{\infty}})(n-1)$.
        Note that, for a fixed $v$, if $k$ is large enough, there is some space left between these two bounds.
        This is what we will make use of in the following section.
        } 
\end{remark}

\section{Spectral invariants}\label{sec:spectral}
We keep the setting and notation of the last section.
We will now define the spectral invariants which we will use to show one side of the quasi-isometry property in Theorem~\ref{thm:main_1}.
\begin{lemma}\label{lem:lower_bound}
Let $v,w \in \Reals^d$ be generic. Then 
\[\frac{1}{2}\norm{v-w}_{\infty} \leq d_{\Hof}(\phi_v(L_2), \phi_w(L_2)).\]
\end{lemma}
\begin{proof}
In keeping with the conventions of the last section, we will only treat the case of $n \geq 2$.
The two-dimensional case (i.e.\ $n=1$) is already implied by~\cite[Theorem 1.5]{zapolsky-2013}\footnote{This theorem applies to a Lagrangian which fibers over $S^1$ and gives an infinite-dimensional Hofer quasi-flat.}.

Since we can choose $v,w \in \Reals^d$ to be generic, we can assume that $\phi_v(L_2) \trans L_0$ as well as $\phi_w(L_2) \trans L_0$.
Note that this also implies that $ \tau^{2k}_i(L_0) \trans \phi_v(L_2)$ and $ \tau^{2k}_i(L_0) \trans \phi_w(L_2)$ for all admissible $k$ and $i$.
We set
\[k \coloneqq 2 \left( \lceil \max\{\norm{v}_\infty, \norm{w}_\infty\} \rceil + 3 \right).\]
Fix $i \in \{0,\dots,d\}$.
Consider an intersection point $c \in \tau^{2k}_i(L_0) \cap \phi_v(L_2)$ such that ${\norm{c^\sharp} \not\in [\hat{h}_i,\check{h}_{i+1}]}$.
The latter implies that $c \not\in \supp \tau^{2k}_i$.
By the results from the last section, we know that the corresponding generator of $CF(\tau^{2k}_i(L_0) ,\phi_v(L_2))$ must either live in a degree 
lower or equal to 
\[n + \floor{2 \norm{v}_{\infty}}(n-1) < n + k(n-1)\]
or a degree above or equal to \begin{align*}
(2k - \floor{2 \norm{v}_{\infty}})(n-1) &\geq 2(k - (\norm{v}_\infty + 1))(n-1) \\ &> 2\left(\frac{k}{2} + 1\right)(n-1)
\geq k(n-1)+n.\end{align*}
Here we use that $(\norm{v}_\infty + 1) < \frac{k}{2} - 1$ and 
 $3(n-1) > n$ since $n \geq 2$.
By the choice of $k$ the same holds true for such generators of $CF(\tau^{2k}_i(L_0) ,\phi_w(L_2))$.
Thus, there is a unique generator of degree {$n + k(n-1)$} in both $CF(\tau^{2k}_i(L_0),\phi_v(L_2))$ and $CF(\tau^{2k}_i(L_0),\phi_w(L_2))$ and this generator lies in the 
radial band from $\hat{h}_i$ to $\check{h}_{i+1}$.
The norm here should again be understood with respect to the local model in $U$.

Similarly, if we use the same gradings for $\tau^{2k}_i(L_0)$ and $L_2$, the homology group $HF(\tau_i^{2k}(L_0),L_2)$
	has a unique class in each of the following degrees by~\cite[Theorem 2.13]{frauenfelder-schlenk-2005} and Remark~\ref{rmk:fs} after applying our grading conventions (see Lemma~\ref{lem:degrees_tau}):
\begin{center}
    \begin{tabular}{ c|c|c|c|c } 
        $n$ & $n + (n-1)$ & $n + 2(n-1)$ &\dots & $n + (2k-1) (n-1)$ 
    \end{tabular}
\end{center} 
The same is true for any $HF(\tau_i^{2k}(L_0),L)$ with $L \in \Lag(L_2)$ since Lagrangian Floer homology is invariant under Hamiltonian isotopies
and the continuation map preserves degrees.
We had already established in the previous section, that the above degrees are the only possible degrees for
the intersection points generated by the Dehn twist. 
We will now make use of this
abundance of homology classes.
We denote by $\alpha_i$ the unique class in degree {$n + k(n-1)$} of the Floer homology group $HF(\tau_i^{2k}(L_0),L_2 )$.
Then for any $i \in \{0,\dots,d-1\}$ we set 
\[a_i(v) \coloneqq c(\phi_v^*\alpha_{i+1}; HF(\tau^{2k}_{i+1}(L_0), \phi_v(L_2))) - c(\phi_v^*\alpha_i; HF(\tau^{2k}_i(L_0), \phi_v(L_2))) \]
and
\[a_i(w) \coloneqq c(\phi_w^*\alpha_{i+1}; HF(\tau^{2k}_{i+1}(L_0), \phi_w(L_2))) - c(\phi_w^*\alpha_i; HF(\tau^{2k}_i(L_0), \phi_w(L_2))), \]
where $c(\beta; HF(L,L'))$ denotes the spectral invariant of the class $\beta$ in $HF(L,L')$.
Due to the index considerations given above, we know that $\phi_v^*\alpha_i$ and $\phi_w^*\alpha_i$ are both supported at a single intersection point which lies in the radial band $[\hat{h}_i,\check{h}_{i+1}]$. 
{This is due to the choice of $k$, which rules out that any intersection point with index $n+k(n-1)$ could lie in the support of $\phi_v$ or $\phi_w$.
Compare the end of Section~\ref{sec:index_comp} to the choice of $k$ made at the beginning of the proof to see this.
Recall, that all generators in the support of $\tau_i^2$ can be enumerated, see Remark~\ref{rmk:intersection_point_count}.
Furthermore, there only is one generator in the support of $\tau_i$ with the correct degree according to Lemma~\ref{lem:degrees_tau}.}
Let us denote this point by $\xi_i \in \tau^{2k}_i(L_0)$.
This implies that there is an easy formula for the above spectral invariants.
They simply correspond to the symplectic action of $\xi_i$.
Since we are looking at a Floer complex defined without a Hamiltonian perturbation, the action can be computed directly from the primitives.
Indeed,
\begin{align*}
    a_i(v) &= c(\phi_v^*\alpha_{i+1}; HF(\tau^{2k}_{i+1}(L_0), \phi_v(L_2))) - c(\phi_v^*\alpha_i; HF(\tau^{2k}_i(L_0), \phi_v(L_2))) \\
    &= \A_{v}(\xi_{i+1}) - \A_v(\xi_i) \\
    &= h_{\phi_v(L_2)}(\xi_{i+1}) -h_{\tau^{2k}_{i+1}(L_0)}(\xi_{i+1}) + h_{\tau^{2k}_i(L_0)}(\xi_i) - h_{\phi_v(L_2)}(\xi_i),
\end{align*}
where $\A_v$ denotes the action functional associated to $CF(\tau^{2k}_i(L_0), \phi_v(L_2))$.

We can now connect these spectral invariants to the Hofer distance in the usual way, i.e.\ by utilizing that 
the spectral invariants are $1$-Lipschitz with respect to Hamiltonian perturbations of the Lagrangians (see Corollary~\ref{cor:lipschitz}).
Assume ${\psi \in \Ham_c(M)}$ is any compactly supported Hamiltonian diffeomorphism with $\psi(\phi_v(L_2)) = \phi_w(L_2)$. Then $\psi$ induces a map $HF(\tau^{2k}_i(L_0), \phi_v(L_2)) \to HF(\tau^{2k}_i(L_0), \phi_w(L_2))$ on Floer homology that is an isomorphism of graded vector spaces. Thus, it has to map $\phi_v^*\alpha$ to $\phi_w^*\alpha$
for degree reasons. The Lipschitz property of the spectral invariants implies that for any class $\beta \in HF(\tau^{2k}_i(L_0), \phi_v(L_2))$ we have 
\[\abs{c(\beta; HF(\tau^{2k}_i(L_0), \phi_v(L_2))) - c(\psi^*\beta; HF(\tau^{2k}_i(L_0), \phi_w(L_2)))} \leq \norm{\psi}_{\Hof}.\]
Combining this with the fact that $\psi$ maps $\phi_v(L_2)$ to $\phi_w(L_2)$, we obtain
\begin{equation*}
\begin{split}
\abs{a_i(v) - a_i(w)} &\leq \abs{c(\phi_v^*\alpha_{i+1}; HF(\tau^{2k}_{i+1}(L_0), \phi_v(L_2))) - c(\phi_w^*\alpha_{i+1}; HF(\tau^{2k}_{i+1}(L_0), \phi_w(L_2)))} \\ &\qquad + \abs{c(\phi_v^*\alpha_{i}; HF(\tau^{2k}_i(L_0), \phi_v(L_2))) - c(\phi_w^*\alpha_{i}; HF(\tau^{2k}_i(L_0), \phi_w(L_2)))} \\ &\leq 2d_{\Hof}(\phi_v(L_2), \phi_w(L_2)).
\end{split}
\end{equation*}
We can compute the quantity on the left-hand side explicitly by using the construction of $\phi_v$ and $\phi_w$. Namely, by using~\eqref{eq:primitive_phi} we obtain that
\begin{align*}
 a_i(v) - a_i(w) &= ( h_{\phi_v(L_2)}(\xi_{i+1})-h_{\tau^{2k}_{i+1}(L_0)}(\xi_{i+1})  + h_{\tau^{2k}_i(L_0)}(\xi_i) - h_{\phi_v(L_2)}(\xi_i))
        \\ &\qquad- ( h_{\phi_w(L_2)}(\xi_{i+1})-h_{\tau^{2k}_{i+1}(L_0)}(\xi_{i+1})  + h_{\tau^{2k}_i(L_0)}(\xi_i) - h_{\phi_w(L_2)}(\xi_i))
        \\ &= h_{\phi_v(L_2)}(\xi_{i+1})-h_{\phi_v(L_2)}(\xi_i)-h_{\phi_w(L_2)}(\xi_{i+1})+h_{\phi_w(L_2)}(\xi_i)
        \\ &= \int_0^{\norm{\xi_i^\sharp}}\theta_v(s)ds - \int_0^{\norm{\xi_{i+1}^\sharp}}\theta_v(s)ds  + \int_0^{\norm{\xi_{i+1}^\sharp}}\theta_w(s)ds 
            - \int_0^{\norm{\xi_i^\sharp}}\theta_w(s)ds
        \\ &= \int_{\norm{\xi_i^\sharp}}^{\norm{\xi_{i+1}^\sharp}}\theta_w(s) - \theta_v(s)ds
        = (w_{i+1}-v_{i+1}) \int_{\norm{\xi_i^\sharp}}^{\norm{\xi_{i+1}^\sharp}}\theta_{i+1}(s) 
        \\ &= w_{i+1}-v_{i+1}.
\end{align*}
Recall that $\theta_{i+1}(\cdot) \coloneqq \theta(\cdot - \check{h}_{i+1})$, which implies that
$\supp \theta_{i+1} \subset \interval{\norm{\xi_i^\sharp}}{\norm{\xi_{i+1}^\sharp}}$.
In the last step we use this fact together with the normalization of $\theta$.
The above estimate together with ${\abs{a_i(v) - a_i(w)} \leq 2d_{\Hof}(\phi_v(L_2), \phi_w(L_2))}$ implies that 
\[\frac{1}{2}\abs{v_{i+1}-w_{i+1}} \leq d_{\Hof}(\phi_v(L_2), \phi_w(L_2)).\]
Since $i \in \{0,\dots,d-1\}$ was arbitrary, this implies that 
\[\frac{1}{2}\norm{v-w}_{\infty} \leq d_{\Hof}(\phi_v(L_2), \phi_w(L_2)).\]
\end{proof}

\section{Finite-dimensional Hofer quasi-flats}\label{sec:flats}
We will now prove a finite-dimensional version of Theorem~\ref{thm:main_1}.
For this, we keep the setting of the previous sections.
Recall, that this means that $(M^{2n},\omega= d\lambda)$ is a Liouville domain with $2c_1(M)=0$.
Further, we assume that $L_0,L_1,L_2 \subset M$ form an $A_3$-configuration of exact Lagrangian spheres.
{As in previous sections, the following proof can be applied verbatim to the cotangent bundle case as well.
For the cotangent case, we work with the notation $M = (T^*S^n, \lambda_{\taut})$ with $L_0 = T_x^*S^n,L_1 = o_{S^n},L_2 = T_y^*S^n$, where $x,y \in S^n$ are distinct
and not antipodal.
Note that this notation is chosen to emphasize that the same proof can be applied to both situations as is.}

Lemma~\ref{lem:lower_bound} already gives us one side of the inequality we need to obtain a Hofer quasi-flat.
The following lemma, which can be proven by direct computation, gives the other:
\begin{lemma}\label{lem:upper_bound}
Let $v,w \in \Reals^d$ be arbitrary. Then 
\begin{align*}d_{\Hof}(\phi_v(L_2), \phi_w(L_2)) \leq 2\norm{v-w}_1 \leq 2d \norm{v-w}_{\infty},\end{align*}
{where $\norm{v}_1 = \abs{v_1}  + \cdots + \abs{v_d}$ is the standard $1$-norm on $\Reals^d$.}
\end{lemma}
\begin{proof}
Let us consider the following family of autonomous Hamiltonians:
\begin{align*}
    H^{v \to w}_i: U &\to \Reals  \\
    (x,\xi_x) &\mapsto \int_0^{\norm{\xi_x^\sharp}} (w_i - v_i)\theta_i(t) dt,
\end{align*}
for $i \in \{1,\dots,d\}$.
They can be extended to all of $M$ by a constant outside $U$.
It is easy to check that 
\[(\phi_{H^{v \to w}_d} \circ \dots \circ \phi_{H^{v \to w}_1})(\phi_v(L_2)) = \phi_w(L_2).\]
By the triangle inequality for the Hofer metric, this thus implies 
\begin{equation}\label{eq:hofer_tri_ineq}d_{\Hof}(\phi_v(L_2), \phi_w(L_2)) \leq \sum_{i=1}^d \osc H^{v \to w}_i.\end{equation}
We note that 
\begin{align*}
    \osc H^{v \to w}_i &\leq 2\max \abs{H^{v \to w}_i} 
    = 2 \max_{c > 0} \abs*{ \int_0^c(w_i - v_i)\theta_i(t) dt} \\
    &\leq 2  \abs{w_i - v_i} \max_{c > 0} \int_0^c \theta_i(t) dt = 2 \abs{v_i - w_i}.
\end{align*}
By plugging this into~\eqref{eq:hofer_tri_ineq} we obtain
\begin{align*}d_{\Hof}(\phi_v(L_2), \phi_w(L_2)) \leq 2\norm{v-w}_1 \leq 2d \norm{v-w}_{\infty}.\end{align*}
\end{proof}
With this lemma in place, we can show the following finite-dimensional analogue of Theorem~\ref{thm:main_1}:
\begin{theorem}\label{thm:aux_1} 
    Let $(M, \omega = d\lambda)$ be a Liouville domain with $2c_1(M) = 0$ 
    and $L_0, L_1, L_2 \subset M$ exact Lagrangian spheres in an $A_3$-configuration. 
    Then, for any $d\in\mathbb{N}$, there is a map 
    \begin{align*}
        \Phi: (\Reals^d, d_\infty) &\into (\Lag(L_2), d_{\Hof})
    \end{align*}
    which is a quasi-isometric embedding with quasi-isometry constant $2d$, i.e.\ for any $v,w \in \Reals^d$
    \[\frac{1}{2d}d_\infty(v,w) \leq d_{\Hof}(\Phi(v),\Phi(w)) \leq 2d \cdot d_{\infty}(v,w).\]
\end{theorem}
\begin{proof}
    If $v,w \in \Reals$ are generic in the sense of Lemma~\ref{lem:lower_bound}, then we already obtain 
    \[\frac{1}{2d}\norm{v-w}_\infty \leq d_{\Hof}(\phi_v(L_2),\phi_w(L_2)) \leq 2d \norm{v-w}_{\infty} \]
    from Lemma~\ref{lem:lower_bound} and Lemma~\ref{lem:upper_bound}.
    For the general case we obtain the result by approximation.
    The assumption of being generic in Lemma~\ref{lem:lower_bound}  means that $\phi_v(L_2) \trans L_0$ and $\phi_w(L_2) \trans L_0$.
    This simply means that $2\pi v_i \not= \pm \delta \mod 2\pi$ and $2\pi w_i \not= \pm \delta \mod 2\pi$ for all $i \in \{1,\dots,d\}$.
    Thus, the set of exceptions is discrete.
    Assume $v^{(k)} \to v$ and $w^{(k)} \to w$ are sequences of generic points converging in the $\norm{\cdot}_\infty$-norm.
    Then by Lemma~\ref{lem:upper_bound}, $\phi_{v^{(k)}}(L_2) \to \phi_v(L_2)$ and $\phi_{w^{(k)}}(L_2) \to \phi_w(L_2)$ in the Hofer metric.
    Further, 
    \begin{align*}
        \frac{1}{2d}\norm{v-w}_\infty &= \lim_{k \to \infty} \frac{1}{2d} \norm{v^{(k)} - w^{(k)}}_\infty \leq \lim_{k \to \infty} d_{\Hof}(\phi_{v^{(k)}}(L_2), \phi_{w^{(k)}}(L_2)) \\ &= d_{\Hof}(\phi_v(L_2),\phi_w(L_2)) \leq 2d \norm{v-w}_{\infty},
    \end{align*}
    where we apply Lemma~\ref{lem:lower_bound} to $v^{(k)}$ and $w^{(k)}$ for all $k \in \Nats$.
\end{proof}
\begin{remark}\label{rmk:infinity_flats} 
    Given that we can find quasi-flats of any finite dimension, the question of a quasi-isometric embedding of $\Reals^\infty$ into $\Lag(L_2)$ is natural to ask.
    The above construction does not give us such an embedding.
    Most obviously, it fails due to Lemma~\ref{lem:upper_bound}.
    Since the constant in Lemma~\ref{lem:upper_bound} diverges as $d \to \infty$, the upper bound fails in the limit.
    However, even if we had a version of Lemma~\ref{lem:upper_bound} with a constant independent of $d$, we would not obtain an infinite-dimensional Hofer flat
    by the method laid out above.
    Indeed, it is not possible to choose $0 < \hat{h}_0 < \check{h}_1 < \hat{h}_1 < \cdots < 1$
    in such a way that the construction from Section~\ref{sec:setup} applied to $0 < \hat{h}_0 < \check{h}_1 < \hat{h}_1 < \cdots < \check{h}_d < \hat{h}_d < \check{h}_{d+1} < 1$ gives compatible results for all $d \in \Nats$.
    Thus, it is not possible to pass to the limit.
    However, both of these problems can be fixed by using a slightly different geometric construction as shown in Section~\ref{sec:infty_flats}.
    One feature of this is that we actually obtain two different families of finite-dimensional Hofer quasi-flats associated to 
    the two methods of construction.
\end{remark}
\begin{remark}
    Finite-dimensional versions of Corollary~\ref{cor:main_1} and Corollary~\ref{cor:main_2} follow by the same proof as in Theorem~\ref{thm:aux_1}.
    Note that Corollary~\ref{cor:main_2} is simply the local version and thus all arguments from Sections~\ref{sec:action_comp},~\ref{sec:index_comp} and~\ref{sec:spectral} apply.
    The case of the $A_2$-plumbing can be obtained similarly.
    We can describe the plumbing $A_2^n$ as the Liouville completion of the manifold obtained by attaching a handle to the boundary sphere 
    of the fiber $D^*_yS^n \subset D^*S^n$ for some $y \in S^n$. We then let $L_1$ be the sphere resulting from this handle-attachment and $L_0$ a fiber over 
    any point $x \in S^n \setminus \{y,-y\}$.
    Note that in the alternative --- more standard --- construction of gluing two disk cotangent bundles together via an identification that switches fiber and zero section in a neighborhood of the gluing point, this $L_1$ corresponds to one of the zero sections.
    The proof laid out above then gives the corresponding result for $A^n_2$, i.e.\ a finite-dimensional version of Corollary~\ref{cor:main_1}.
\end{remark}

\section{Infinite-dimensional Hofer quasi-flats}\label{sec:infty_flats}
We will now construct infinite-dimensional Hofer flats following up on Remark~\ref{rmk:infinity_flats}.
For this we need to choose data as in Section~\ref{sec:setup} in a compatible way for all $d \in \Nats$.
Choose $0 = \check{h}_0 < \hat{h}_0 < \check{h}_1 < \hat{h}_1 < \cdots < 1$.
In this case $\hbar$ would be $0$, so that we have to slightly modify the procedure 
given in Section~\ref{sec:setup}.
For any $i \in \Nats$ we choose a bump function $\theta_i \in C^\infty(\Reals,\interval[open right]{0}{\infty})$ such that 
\begin{enumerate}
    \item the support of $\theta_i$ is contained in $\interval[open]{\check{h}_i}{\hat{h}_i}$;
    \item the function $\theta_i$ has exactly one local maximum with value $2\pi$ and its derivative does not vanish elsewhere on the interior of its support;
    \item $\int_{\Reals} \theta_i(t) dt = 1$.
\end{enumerate}
For any $d \in \Nats$ we construct a group homomorphism $\Phi^{(d)}: (\Reals^d,+) \to  (\Ham_c(M), \circ)$ as laid out in Section~\ref{sec:setup}.
By using compatible data as described above, 
we can achieve that for any $m > 0$, $\Phi^{(d+m)}((v_1,\dots,v_d,0,\dots,0)) = \Phi^{(d)}(v)$ for any $v \in \Reals^d$.
Thus, by taking a direct limit, we obtain 
\begin{align*}
        \Phi^{(\infty)}: \Reals^\infty &\into \Ham_c(M) \\
        v &\mapsto \phi_v
\end{align*}
as a well-defined map.
It remains to adapt the construction of our preferred Dehn twists.
For this, we choose a function $\rho_i \in C^\infty(\Reals,\interval[open right]{0}{\infty})$ for any $i \in \Nats_0$ such that 
\begin{enumerate}
    \item $\rho_i$ is monotone;
    \item the function $\rho_i$ is equal to $\pi$ on an open interval that contains $\interval[open left]{-\infty}{\check{h}_i}$;
    \item the function $\rho_i$ vanishes on an open interval that contains $\interval[open right]{\hat{h}_{i+1}}{\infty}$;
\end{enumerate}
Then we denote the model Dehn twist defined by using $\rho_i$ by $\tau_i: M \to M$ for any $i \in \Nats_0$.
As in Section~\ref{sec:setup}, the supports of $\phi_v$ and $\tau_0^2, \dots$ are pairwise disjoint for any $v \in \Reals^\infty$.

Upon inspection of the proof of Lemma~\ref{lem:lower_bound}, we note that the above properties of $\theta_1,\dots,\theta_d$ and $\rho_0,\dots,\rho_d$
are sufficient for the argument.
Thus, since for any $v,w \in \Reals^\infty$ there is some $D \in \Nats$ with $v,w \in \Reals^D$, we can apply Lemma~\ref{lem:lower_bound}
to obtain the following corollary:
\begin{corollary}\label{cor:infty_lower_bound}
Let $v,w \in \Reals^\infty$ be generic. Then 
\[\frac{1}{2}\norm{v-w}_{\infty} \leq d_{\Hof}(\Phi^{(\infty)}(v)(L_2),\Phi^{(\infty)}(w)(L_2)).\]
\end{corollary}

\begin{figure}[ht] 
    \centering 
    \includegraphics[width=0.95\textwidth]{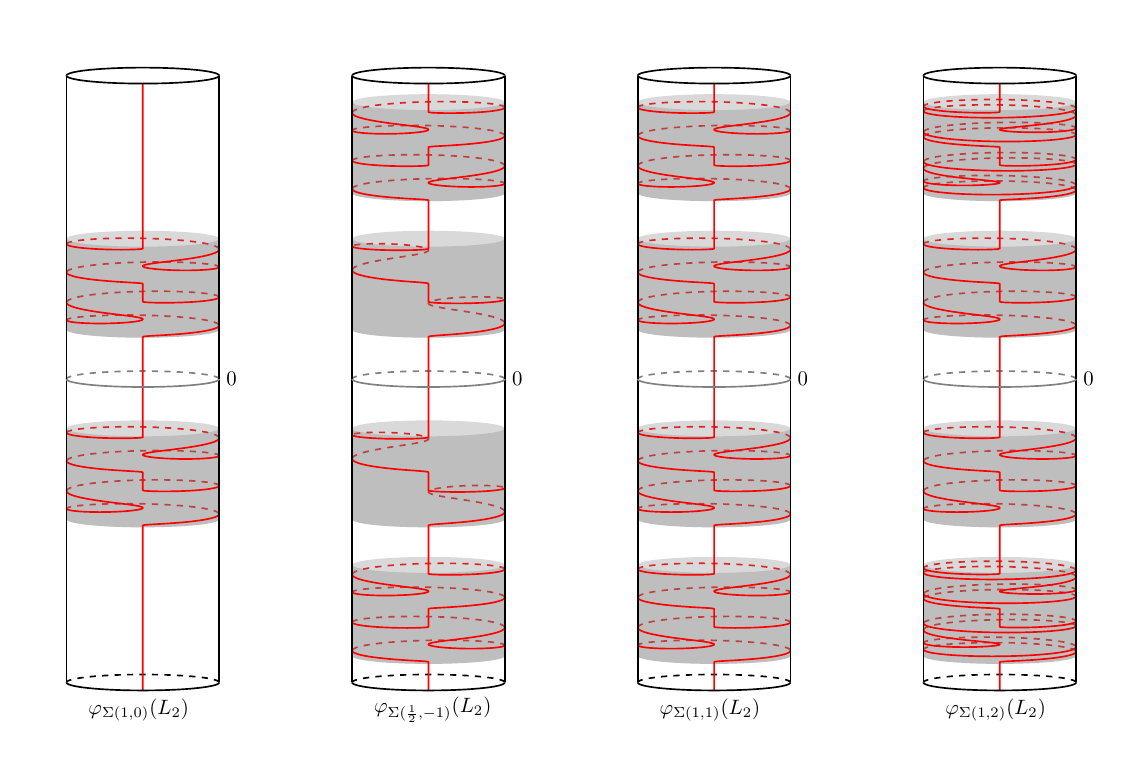} 
    \caption{Sketch of Lagrangians obtained by applying $(\Phi^{(\infty)} \circ \Sigma)(v)$ to $L_2$ for different $v \in \Reals^2 \subset \Reals^\infty$.
    The support of the Hamiltonian \textit{function} generating the diffeomorphism is marked in gray. Note that the different regions now have 
    disjoint support (of the Hamiltonian function not just the diffeomorphism) which was not the case for the earlier construction.}\label{fig:sigma_visualization}
\end{figure}
The proof of Lemma~\ref{lem:upper_bound} however, as mentioned before, cannot be directly adapted since $\norm{\cdot}_1$ and $\norm{\cdot}_\infty$ are not 
equivalent norms on $\Reals^\infty$.
By using a trick, we can rectify this and give a --- slightly different --- construction for $\infty$-dimensional flats. 
For this we first define the following map:
\begin{align*}
\Sigma: \Reals^\infty &\to \Reals^\infty \\
v &\mapsto (v_1,-v_1,v_2,-v_2,\dots).
\end{align*}
We will later construct an explicit Hamiltonian isotopy from $(\Phi^{(\infty)} \circ \Sigma)(v)(L_2)$ to $(\Phi^{(\infty)} \circ \Sigma)(w)(L_2)$.
Since neighboring radial bands are controlled by $v_{i}$ and $-v_i$, 
we can decompose the Hamiltonian function generating this diffeomorphism into different Hamiltonians with pairwise disjoint support.
See Figure~\ref{fig:sigma_visualization} for a visualization of this idea. We also note that with $\Sigma$ applied, the Hamiltonian diffeomorphism
$(\Phi^{(\infty)} \circ \Sigma)(v)$ is similar to the reparametrizations of the geodesic flow used in~\cite{usher-2014}.
Using this strategy, we can show the following lemma, which directly implies Theorem~\ref{thm:main_1}. The proof of Corollary~\ref{cor:main_1} and Corollary~\ref{cor:main_2} is completely analogous in the respective setting.
{Note that the local model and thus the index and action computations apply to these settings directly. Since Lagrangian Floer homology is well-defined in both cases (see Section~\ref{sec:prelims}),
the proof can be applied verbatim.}
\begin{lemma}
For any $v,w \in \Reals^\infty$ we have 
\[\frac{1}{2}\norm{v-w}_\infty \leq d_{\Hof}((\Phi^{(\infty)} \circ \Sigma)(v)(L_2),(\Phi^{(\infty)} \circ \Sigma)(w)(L_2))
\leq 2\norm{v-w}_\infty.\]
\end{lemma}
\begin{proof}
    The left-hand side is easy to show.
    Note that $\norm{\Sigma v}_\infty = \norm{v}_\infty$.
    Thus, we can apply Corollary~\ref{cor:infty_lower_bound} together with an approximation argument as in the proof of Theorem~\ref{thm:aux_1}.

    It remains to show the right-hand inequality.
    For this we fix $v,w \in \Reals^\infty$ for the remainder of the proof.
    Let $d\in\Nats$ be chosen such that $v,w \in \Reals^d$.
    Let us again start by considering the following family of autonomous compactly supported Hamiltonians:
    \begin{align*}
        H^{v \to w}_i: U &\to \Reals  \\
        (x,\xi_x) &\mapsto \int_0^{\norm{\xi_x^\sharp}} (w_i - v_i)\theta_{2i-1}(t) - (w_i - v_i)\theta_{2i}(t) dt,
    \end{align*}
    for $i \in \{1,\dots,d\}$.
    They can be extended to $M$ by $0$ outside $U$.
    Note that any $H^{v \to w}_i$ is supported in the interior of the radial band $\interval{\check{h}_{2i-1}}{\hat{h}_{2i}}$, where the norm is taken with respect to the local model in $U$.
    We can explicitly compute the oscillation of $H^{v \to w}_i$ as 
    \[\osc H^{v \to w}_i = \abs{w_i-v_i},\]
    where we use that $\theta_{2i-1}$ and $\theta_{2i}$ have disjoint support and that $\int_{\Reals} \theta_{2i-1} = \int_\Reals \theta_{2i} = 1$.
    We can then define $H^{v \to w} \coloneqq H^{v \to w}_1 + \cdots + H^{v \to w}_d$.
    By construction 
    \[\phi_{H^{v \to w}}((\Phi^{(\infty)} \circ \Sigma)(v)(L_2)) = (\Phi^{(\infty)} \circ \Sigma)(w)(L_2).\]
    Now since all $H^{v \to w}_1 , \dots , H^{v \to w}_d$ have pairwise disjoint support, we have 
    \begin{align*}
        \osc H^{v \to w} &\leq 2\max_{i \in \{1,\dots,d\}}\left( \osc {H^{v \to w}_i} \right)
        = 2 \max_{i \in \{1,\dots,d\}} \left( \abs{v_i-w_i}\right)\\
        &= 2 \norm{v-w}_\infty.
    \end{align*}
    Thus, $d_{\Hof}((\Phi^{(\infty)} \circ \Sigma)(v)(L_2),(\Phi^{(\infty)} \circ \Sigma)(w)(L_2))
    \leq 2\norm{v-w}_\infty$ and the proof is complete since $v,w \in \Reals^\infty$ were arbitrary.
\end{proof}
\begin{remark}
To conclude Corollary~\ref{cor:absolute_version} it is sufficient to notice two things. 
Firstly, the quasi-isometric embedding is of course given by $\Psi \coloneqq \Phi^{(\infty)} \circ \Sigma: \Reals^\infty \into \Ham_c(M)$. We notice immediately that 
for any $v,w \in \Reals^\infty$, 
\[\frac{1}{2}\norm{v-w}_{\infty} \leq d_{\Hof}(\Psi^{(\infty)}(v)(L_2),\Psi^{(\infty)}(w)(L_2)) \leq \norm{\Psi^{(\infty)}(w)(\Psi^{(\infty)}(v))^{-1}}_{\Hof}.\] 
Secondly, notice that the Hamiltonian $H^{v \to w}$ constructed above actually generates $\Psi^{(\infty)}(w)(\Psi^{(\infty)}(v))^{-1}$.
Thus, by the computation from the last proof, \[\norm{\Psi^{(\infty)}(w)(\Psi^{(\infty)}(v))^{-1}}_{\Hof} \leq 2\norm{v-w}_\infty,\]
which implies Corollary~\ref{cor:absolute_version}.
\end{remark}

\section{Applications to the boundary depth}\label{sec:boundary_depth}
Given that we computed all indices and actions for generators of $CF(\tau_i^k(L_0), \phi_v(L_2))$, 
we can also make some conclusions about the boundary depth of this complex.
{In this section we will only need the embedding $\Phi: \Reals^1 \into \Lag(L_2)$ constructed in the proof of Theorem~\ref{thm:aux_1}, i.e.\ the $d=1$ case.}
So, we fix some data as in Section~\ref{sec:setup} and keep the notation of the previous sections.
We denote $\Phi(k)$ by $\phi_{(k)}$ consistent with the $\phi_v$ notation used in the previous sections.
We will now consider the specific Floer complex $CF(L_0, \phi_{(k)}(L_2))$ for all $k \in \Nats$. 
First, we want to note that we can describe the intersection points $L_0 \cap \phi_{(k)}(L_2)$
explicitly.
Indeed, they are characterized by the fact that for $\xi \in L_0 \cap \phi_{(k)}(L_2)$ we have
\begin{align}
    \label{eq:intersection_cond_1}k \cdot \theta_{1}(\norm{\xi^\sharp}) &= \delta + 2\pi\cdot m \\ 
    \label{eq:intersection_cond_2}k \cdot \theta_{1}(\norm{\xi^\sharp}) &= 2\pi - \delta + 2\pi\cdot m,
\end{align}
{where $\theta_v$ is the bump function used in the construction of $\phi_v$ as in Section~\ref{sec:setup}. We keep the notation from that section.
Recall that $\theta_1$ is a shorthand for $\theta(\cdot - \check{h}_1)$.
As before, $\delta$ denotes the distance between the intersection points in $L_1 \cap L_0$ and $L_1 \cap L_2$
with respect to the round metric induced by the framing of $L_1$.
}
The norms above are again taken with respect to the round metric in the local model induced by the framing of $L_1$.
Note that both for~\eqref{eq:intersection_cond_1} and~\eqref{eq:intersection_cond_2} we want to allow $m = 0$.
We consider both equations as equations in a formal variable $t$, i.e.\ 
{
\begin{align*}
    k \cdot \theta_{1}(t) &= \delta + 2\pi\cdot m \\ 
    k \cdot \theta_{1}(t) &= 2\pi - \delta + 2\pi\cdot m.
\end{align*}}
By the choice of $\theta$ we know that both equations will each have $2k$ {solutions for t.} 
Due to the range of $\theta$, such a solution will exist {exactly} for $m \in \{0,\dots,k-1\}$.
Of those there will be $k$ solutions with values below $\check{h}_1 + \frac{\hbar}{2}$ and $k$ with values above $\check{h}_1 + \frac{\hbar}{2}$.
Denote the solution to $k\theta_{1}(t) = \delta + 2\pi \cdot m$ with $t <\check{h}_1 + \frac{\hbar}{2}$ by $\check{t}^+_{k,m}$.
Similarly, we denote the solutions with $t > \check{h}_1 + \frac{\hbar}{2}$ by $\hat{t}^+_{k,m}$.
Thus, we get \[\check{t}^+_{k,0} < \cdots < \check{t}^+_{k,k-1} < \check{h}_1 + \frac{\hbar}{2} < \hat{t}^+_{k,k-1} < \cdots < \hat{t}^+_{k,0}.\]
For the solutions of $k\theta_{1}(t) = 2\pi - \delta + 2\pi\cdot m$ we follow the same naming scheme. Namely, we obtain the solutions 
\[\check{t}^-_{k,0} < \cdots < \check{t}^-_{k,k-1} <  \check{h}_1 + \frac{\hbar}{2}< \hat{t}^-_{k,k-1} < \cdots < \hat{t}^-_{k,0},\]
which satisfy {$k\theta_{1}(\hat{t}^-_{k,m}) = k \theta_{1}(\check{t}^-_{k,m}) = 2\pi - \delta + 2\pi m$.}
See Figure~\ref{fig:ts} for a visualization.
{These formal solutions correspond to actual intersection points.
Recall that we denote the intersection point of $L_0$ and $L_1$ by $x$ and the intersection point of $L_1$ and $L_2$ by $y$.
Let $\gamma: \interval{0}{\delta} \to L_1$ be the unit speed geodesic on $L_1$ (endowed with the round metric as before) going from $y$ to $x$.
Then $\{\pm \check{t}^\pm_{k,m}\dot\gamma(0)\}_{m \in \{0,\dots,k-1\}}$ and $\{\pm \hat{t}^\pm_{k,m}\dot\gamma(0)\}_{m \in \{0,\dots,k-1\}}$ correspond to intersection points $L_0 \cap \phi_v(L_2)$.
We denote the intersection point corresponding to $\check{t}^\pm_{k,m}$ by $\check{c}^\pm_{k,m}$ and 
the intersection point corresponding to $\hat{t}^\pm_{k,m}$ by $\hat{c}^\pm_{k,m}$.
Note in particular, that $\norm{(\check{c}^\pm_{k,m})^\sharp} = \check{t}^\pm_{k,m}$ and $\norm{(\hat{c}^\pm_{k,m})^\sharp} = \hat{t}^\pm_{k,m}$.}
\begin{figure}[ht]
    \centering 
    \includegraphics[width=0.9\textwidth]{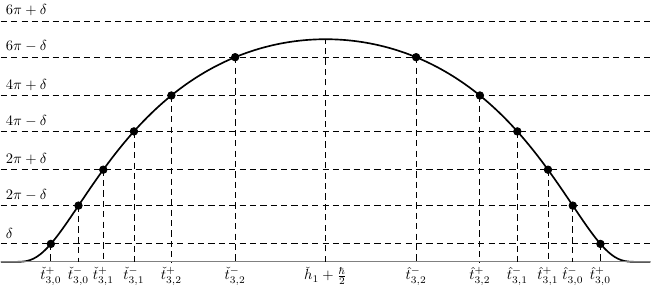}
    \caption{An illustration of the ordering of the $\check{t}^\pm_{k,m}$ and $\hat{t}^\pm_{k,m}$ for $k = 3$.}\label{fig:ts}
\end{figure}

After these preparatory remarks, we can show the following lemma which directly implies Theorem~\ref{thm:main_3}.
\begin{lemma}\label{lem:boundary_depth_unbounded}
    With the notation established above, $\beta(HF(\tau_0^{2\ell}(L_0), \phi_{(k)}(L_2))) \to \infty$ as $k \to \infty$ where $\ell \in \Nats_0$ is fixed.  
\end{lemma}
\begin{proof}
    {First, we note that by construction $\tau_0^{2\ell}$ is the identity on the radial shell from $0$ to $\hat{h}_0$ for any $\ell \in \Nats_0$.
    This follows directly from the fact that $\tau_0$ is the antipodal map in this region.
    It is also the identity on the radial shell outwards from $\check{h}_1$.
    Thus, the intersection points constructed above persist into this setting. 
    
    To begin with the main part of the proof, we first note the following fact:}
    Since $\theta'\vert_{\interval[open]{\iota}{\frac{\hbar}{2}}} > 0$ and $\theta'\vert_{\interval[open]{\frac{\hbar}{2}}{\hbar-\iota}} < 0$ we have $\check{t}^{\pm}_{k,0} > \check{t}^{\pm}_{k+1,0}$ and $\hat{t}^{\pm}_{k+1,0} > \hat{t}^{\pm}_{k,0}$ for all $k \in \Nats$.
    In particular, $\check{t}^{\pm}_{1,0} > \check{t}^{\pm}_{k,0}$ and $\hat{t}^{\pm}_{k,0} > \hat{t}^{\pm}_{1,0}$ for all $k \in \Nats$.

    Assume that $\partial{\check{c}^{+}_{k,0}} = \hat{c}^{+}_{k,0}$.
    Under this assumption the barcode of $HF(\tau_0^{2\ell}(L_2), \phi_{(k)}(L_0))$ must contain a bar of length 
    $\A(\check{c}^{+}_{k,0}) - \A(\hat{c}^{+}_{k,0})$. 
    This implies \[\beta(HF(\tau_0^{2\ell}(L_2), \phi_{(k)}(L_0))) \geq \A(\check{c}^{+}_{k,0}) - \A(\hat{c}^{+}_{k,0}).\]
    Given that we have explicit primitives for $\lambda$ on the relevant Lagrangians, we can compute this bound explicitly.
    By using~\eqref{eq:primitive_phi} we obtain {that 
    \begin{align*}
        \A(\check{c}^{+}_{k,0}) - \A(\hat{c}^{+}_{k,0}) &= h_{\phi_{(k)}(L_2)}(\check{c}^{+}_{k,0}) - h_{\tau_0^{2\ell}(L_0)}(\check{c}^{+}_{k,0})
        - (h_{\phi_{(k)}(L_2)}(\hat{c}^{+}_{k,0}) - h_{\tau_0^{2\ell}(L_0)}(\hat{c}^{+}_{k,0})) \\
        &=  \delta (\check{t}^{+}_{k,0} - \hat{t}^{+}_{k,0}) +  k \int_{\check{t}^+_{k,0}}^{\hat{t}^+_{k,0}}\theta_{1}(s)ds 
        + (h_{\tau_0^{2\ell}(L_0)}(\hat{c}^{+}_{k,0})- h_{\tau_0^{2\ell}(L_0)}(\check{c}^{+}_{k,0})).
    \end{align*}
    Recall that $\rho_0\vert_{\interval[open right]{\check{h}_1}{\infty}} = 0$ and thus by~\eqref{eq:primitive_tau} 
    \begin{align*}
        h_{\tau_0^{2\ell}(L_0)}(\hat{c}^{+}_{k,0})- h_{\tau_0^{2\ell}(L_0)}(\check{c}^{+}_{k,0})
        = f(\hat{c}^{+}_{k,0}) - f(\check{c}^{+}_{k,0}) \geq C,
    \end{align*}
    where $C \coloneqq -2\max\{\abs{f(\xi)} \mid \xi \in \supp \phi_{(1)}\}$.
    Here we use that, since $\hat{t}^{+}_{k,0} > \check{t}^{+}_{k,0} > \check{h}_1$, all 
    other terms in~\eqref{eq:primitive_tau} are equal for both and thus cancel out.}
    
    Note that by construction $\theta_{1}\vert_{\interval{\check{t}^+_{k,0}}{\hat{t}^+_{k,0}}} \geq \delta$.
    {Thus, we can see from the above expression that $\A(\check{c}^{+}_{k,0}) - \A(\hat{c}^{+}_{k,0})$ is always strictly positive for large enough $k$.}
    That is not quite enough for the desired conclusion yet.
    By using $\check{t}^{\pm}_{1,0} > \check{t}^{\pm}_{k,0}$ and $\hat{t}^{\pm}_{k,0} > \hat{t}^{\pm}_{1,0}$ for all $k \in \Nats$,
    we can make the following estimate:
    \begin{align*}
        \A(\check{c}^{+}_{k,0}) - \A(\hat{c}^{+}_{k,0}) 
        &\geq  k \int_{\check{t}^+_{k,0}}^{\hat{t}^+_{k,0}}\theta_{1}(s)ds - \delta (\hat{t}^{+}_{k,0} - \check{t}^{+}_{k,0}) + C \\
        &\geq k \int_{\check{t}^+_{1,0}}^{\hat{t}^+_{1,0}}\theta_{1}(s)ds - \delta (\hat{t}^{+}_{1,0} - \check{t}^{+}_{1,0}) + C.
    \end{align*}
    Clearly, this implies the conclusion, i.e. $\beta(HF(\tau_0^{2\ell}(L_2), \phi_{(k)}(L_0))) \to \infty$ as $k \to \infty$.

    Thus, it remains to show that $\partial\check{c}^{+}_{k,0} = \hat{c}^{+}_{k,0}$.
    By Lemma~\ref{lem:degrees_phi} we know that {
    \begin{align*}
        \mu(\check{c}^+_{k,j}) &= n + (2\ell+2j)(n-1) & \mu(\hat{c}^+_{k,j}) &= n + (2\ell+2j)(n-1)- 1 \\
        \mu(\check{c}^-_{k,j}) &= n + (2\ell+2j+1)(n-1) & \mu(\hat{c}^+_{k,j}) &= n + (2\ell+2j+1)(n-1)- 1 
    \end{align*}}
    Lemma~\ref{lem:degrees_tau} further implies that all other intersection points have a degree {below or equal to $n + (2\ell - 1)(n-1)$ if $\ell \geq 1$.
    If $\ell = 0$ there are no other intersection points.}

    We first consider the case $n=1$. The possible pseudoholomorphic disks in that case are easy to understand.
    We can see from the open mapping theorem, that such disks can only connect $\hat{c}^{+}_{k,m}$ with $\check{c}^{+}_{k,m}$ and
    $\hat{c}^{-}_{k,m}$ with $\check{c}^{-}_{k,m}$ for $m \in \{0,\dots,k-1\}$. See Figure~\ref{fig:holomorpic_disks} 
    for a visualization.
    By counting these disks, we directly see that $\partial\check{c}^{+}_{k,0} = \hat{c}^{+}_{k,0}$.
    {Note in particular that if $\ell \not= 0$ the intersection points coming from the Dehn twist are not connected to any $\hat{c}^{\pm}_i$ or $\check{c}_i^\pm$ by 
    pseudoholomorphic disks.}
    \begin{figure}[ht]
        \centering 
        \includegraphics[width=\textwidth]{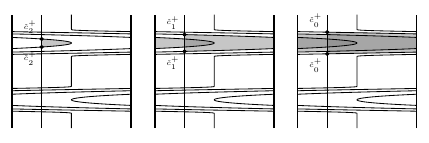}
        \caption{Some holomorphic disks with boundary on $L_0$ and $\phi_{(3)}(L_2)$. All other holomorphic disks are analogous.}\label{fig:holomorpic_disks}
    \end{figure}

    {
    Next we assume $n=3$.
    Here we can employ a symmetry argument based on that used in~\cite[Sec. 2]{frauenfelder-schlenk-2005} to reduce to the $n=1$ case.
    In the following we identify $L_1 \cong S^2 \cong \C P^1$.
    The real locus $\Reals P^1$ of $\C P^1$ is the unique great circle on $\C P^1$ passing through $1$ and $\infty$.
    In particular, it is totally geodesic, i.e. a geodesic through two points on $\Reals P^1 \subset \C P^1$ stays within $\Reals P^1$.
    By~\cite[Lemma 2.14]{frauenfelder-schlenk-2005} we can assume $x,y \in \Reals P^1$ without loss of generality.
    Now consider the map $\sigma: \C P^1 \to \C P^1$ given by $[z_0,z_1] \mapsto [\bar{z}_0:\bar{z}_1]$, i.e. induced by complex conjugation.
    Then this map is an isometry of $\C P^1$ and lifts to a symplectic involution of $T^*\C P^1$,
    which we also denote by $\sigma$.
    Note that $\Reals P^1$ is the fixed point set of $\sigma$, as is $T^*\Reals P^1 \subset T^*\C P^1$ for its lift.
    
    Recall, that a neighborhood $U$ of $L_1$ is identified with some disk cotangent bundle $D^*_rL_1 \cong D_r^* \C P^1$.
    By~\cite[Prop. 2.17]{frauenfelder-schlenk-2005} there is a generic set of $\sigma$-invariant regular almost-complex structures
    on $D_r^* \C P^1$ for any $r > 0$.
    These 
    almost-complex structures also have the property that their restriction to the fixed points of $\sigma$, i.e. to $T^*\Reals P^1$,
    is again a regular almost-complex structure.
    We fix such an almost-complex structure $J$ and denote the induced almost-complex structure on $T^*\Reals P^1$
    by $J^\sigma$.
    In order to define the Floer complex $CF(\tau_0^{2\ell}(L_0), \phi_{(k)}(L_2))$ we extend $J$ outside $U$ in an arbitrary fashion.

    Note that all intersection points $\tau_0^{2\ell}(L_0) \cap \phi_{(k)}(L_2)$
    are contained in some compact subset of $U$.
    Further, the intersection points 
    all lie in $T^* \Reals P^1$ in this local model.
    This is due to the way in which $\tau_0$ and $\phi_{(k)}$ are defined using the cogeodesic flow and the fact that $\Reals P^1 \subset \C P^1$
    is a totally geodesic submanifold.
    Thus, any $c \in {\tau_0^{2\ell}(L_0) \cap \phi_{(k)}(L_2)}$
    satisfies $\sigma(c) = c$.
    Further, note that 
    \begin{align*}
        (\tau_0^{2\ell}(L_0))^\sigma = {\tau_0^{2\ell}(L_0) \cap U \cap T^*\Reals P^1} = D_r^*\Reals P^1 \cap \tau_0^{2\ell}(T^*_x \Reals P^1),\\
        (\phi_{(k)}(L_2))^\sigma = {\phi_{(k)}(L_2)\cap U \cap T^*\Reals P^1} = D_r^*\Reals P^1 \cap \phi^{(k)}(T^*_y \Reals P^1),
    \end{align*}
    where $L^\sigma \coloneqq U \cap L \cap \sigma(L \cap U)$ is the part of the respective Lagrangians inside $U$ that is fixed by $\sigma$.
    The construction of $\tau_0$ and $\phi$ in different dimensions is compatible in such a way that we could also obtain the above 
    by constructing $\tau_0$ and $\phi$ in $T^*\Reals P^1$ instead of in the ambient manifold.
    Thus, when computing the Floer complex $CF((\tau_0^{2\ell}(L_0))^\sigma, (\phi_{(k)}(L_2))^\sigma)$ with respect to $J^\sigma$, we are working exactly in the previous $n=1$ setting.

    By the standard maximum principle argument (see~\cite[Lemma 2.6]{frauenfelder-schlenk-2005})
    finite energy $J$-holomorphic curves connecting intersection points in $\tau_0^{2\ell}(L_0) \cap \phi_{(k)}(L_2)$ must stay within $U$.
    The way we extend $J$ outside $U$ is immaterial to this argument.
    Assume $c_1,c_2 \in \tau_0^{2\ell}(L_0) \cap \phi_{(k)}(L_2)$
    and $u \in \M(c_1,c_2)$, i.e.\ that $u$ is a Floer trajectory connecting $c_1$ and $c_2$.
    If $u$ is $\sigma$-invariant then it must lie in $T^*\Reals P^1$ and 
    gives rise to a $J^\sigma$-holomorphic curve in $T^*\Reals P^1$ 
    connecting $c_1$ and $c_2$.
    If $u$ is not $\sigma$-invariant then (by the $\sigma$ invariance of $J$)
    $\sigma \circ u$ is another $J$-holomorphic curve.
    Since $c_1,c_2$ are fixed by $\sigma$, we obtain $\sigma \circ u \in \M(c_1,c_2)$.
    Thus, any non-$\sigma$-invariant curves in the Floer boundary have to come in pairs and thus cancel each other out since we work with $\field_2$-coefficients.
    It follows that $\partial$ is completely determined by the $J^\sigma$-holomorphic curves in $T^*\Reals P^1$.
    However, this setting corresponds exactly to the $n=1$ case as discussed above.
    It is easy to see that $\check{c}^{+}_{k,0}, \hat{c}^{+}_{k,0} \in \tau_0^{2\ell}(L_0) \cap \phi_{(k)}(L_2)$ 
    correspond to the generators of the same name in the $n=1$ picture.
    Note that by the index formula given above, the index difference between these two generators is $1$ in both the $n=2$ and $n=1$ case.
    The argument used to show $\partial\check{c}^{+}_{k,0} = \hat{c}^{+}_{k,0}$ for the $n=1$ case in the last section thus applies again.
    Since we ruled out any non-trivial contributions to $\partial$ by other curves, 
    this implies $\partial\check{c}^{+}_{k,0} = \hat{c}^{+}_{k,0}$ as claimed.
    }

    Last, we assume that $n > 2$. In this case we obtain the result purely for degree reasons. 
    Recall that the Floer homology $HF(\tau_0^{2\ell}(L_0),\phi_{(k)}(L_2))$ is $\field_2$ in exactly in degrees $n, {n + (n-1)}, {n + 2(n-1)}, \dots, {n + (2\ell-1)(n-1)}$ and vanishes in all other degrees. 
    By Lemma~\ref{lem:degrees_phi} we know that $\check{c}^{+}_{k,0}$ lives in degree $n+2\ell(n-1)$.
    All terms appearing in $\partial\check{c}^{+}_{k,0}$ thus have to live in degree $n+2\ell(n-1)-1$.
    Since the homology of the complex has to vanish in degree ${n+2\ell(n-1) > n+(2\ell-1)(n-1)}$, this boundary must be non-trivial.
    Using the assumption that $n-1>1$ and Lemma~\ref{lem:degrees_phi} we conclude that there is only a single generator of the correct degree, namely $\hat{c}^{+}_{k,0}$.
    Thus, $\partial\check{c}^{+}_{k,0} = \hat{c}^{+}_{k,0}$ as claimed.
\end{proof}

\bibliographystyle{alpha}
\bibliography{refs}

\end{document}